%% file: sparseHW.tex
\documentclass[hidelinks,12pt]{article}
\RequirePackage[OT1]{fontenc}
\RequirePackage{amsthm,amsmath}
\usepackage{amsfonts}
\usepackage{amstext}
\usepackage{parskip}
\usepackage{fullpage}

\usepackage{times}
\usepackage{graphicx}
\usepackage{epsfig}
\usepackage{amsthm,amsmath,amssymb}
\usepackage{graphicx}
\usepackage{epsfig}

\usepackage{amsthm,amsmath,amssymb}
\usepackage{amsfonts}
\usepackage{amstext}
\usepackage{parskip}
\usepackage{fullpage}

\usepackage{times}
\usepackage{graphicx}
\usepackage{epsfig}

\usepackage{graphicx}
\usepackage{epsfig}

{\end{eqnarray}\end{subequations}\hskip-4.0pt}

\def\fatnorm#1{|\kern-.2ex|\kern-.2ex| #1 |\kern-.2ex|\kern-.2ex|}
\newcommand{\twonorm}[1]{\left\lVert#1\right\rVert_2}

\newcommand{\fnorm}[1]{\left\lVert#1\right\rVert_F}
\newcommand{\norm}[1]{\left\lVert#1\right\rVert}
\newcommand{\abs}[1]{\left\lvert#1\right\rvert}

\newcommand{\cov}{\textsf{Cov}}

\newcommand{\off}{{\rm offd}}

\newcommand{\half}{\ensuremath{\frac{1}{2}}}
\newcommand{\inv}[1]{\frac{1}{#1}}
\def\conv{\mathop{\text{\rm conv}\kern.2ex}}

\newcommand{\ip}[1]{\;\langle{\,#1\,}\rangle\;}

\newcommand{\onenorm}[1]{\ensuremath{\left|#1\right|_1}}

\newcommand{\expct}[1]{\ensuremath{\mathbb E}#1}
\newcommand{\silent}[1]{}

\newcommand{\ve}{\varepsilon}

\def\qed{\hskip1pt $\;\;\scriptstyle\Box$}
\def\Ber{\mathop{\text{Bernoulli}\kern.2ex}}
\def\supp{\mathop{\text{supp}\kern.2ex}}
\def\corr{\mathop{\text{corr}\kern.2ex}}
\def\prec{\mathop{\text{precision}\kern.2ex}}
\def\recall{\mathop{\text{recall}\kern.2ex}}
\def\cov{\mathop{\text{Cov}\kern.2ex}}
\def\mnorm{\mathcal{N}_{f,m}\kern.2ex}
\def\var{\mathop{\text{Var}\kern.2ex}}
\def\ess{\mathop{\text{ess}\kern.2ex}}
\def\dom{\mathop{\text{dom}\kern.2ex}}
\def\lin{\mathop{\text{lin}\kern.2ex}}

\newcommand{\func}[1]{\ensuremath{\mathrm{#1}}}
\newcommand{\diag}{\func{diag}}
\newcommand{\offd}{\func{offd}}

\let\tilde\widetilde

\newcommand{\vecp}{\ensuremath{{\bf p}}}

\newcommand{\tr}{{\rm tr}}

\def\E{{\mathbb E}}

\def\supp{\mathop{\text{\rm supp}\kern.2ex}}
\def\argmin{\mathop{\text{arg\,min}\kern.2ex}}

\newcommand{\prob}[1]{\ensuremath{\mathbb P}\left(#1\right)}

\newcommand{\beq}{\begin{equation}}
\newcommand{\eeq}{\end{equation}}
\newcommand{\ben}{\begin{eqnarray}}
\newcommand{\een}{\end{eqnarray}}
\newcommand{\bnum}{\begin{enumerate}}
\newcommand{\enum}{\end{enumerate}}
\newcommand{\bit}{\begin{itemize}}
\newcommand{\eit}{\end{itemize}}
\newcommand{\bens}{\begin{eqnarray*}}
\newcommand{\eens}{\end{eqnarray*}}

\newcommand{\R}{{\bf R}}

\newcommand{\f}{\tilde{f}}

\newtheorem{theorem}{Theorem}[section]

\newtheorem{lemma}[theorem]{Lemma}

\newtheorem{remark}[theorem]{Remark}
\newtheorem{corollary}[theorem]{Corollary}

\def\qed{\hskip1pt $\;\;\scriptstyle\Box$}

\newenvironment{proofof}[1]{\hspace*{20pt}{\it Proof}{ of #1}.\hskip10pt}{\qed\vskip5pt}
\newenvironment{proofof2}{\hskip10pt}{\qed\vskip5pt}

\begin{document}
\title{Sparse Hanson-Wright inequalities for subgaussian quadratic  forms}
\author{Shuheng Zhou \\
Department of Statistics, University of Michigan, Ann Arbor, MI  \\
TR539, October 2015\\[20pt]
This Version: }

\maketitle

\begin{abstract}

In this paper, we provide a proof for the Hanson-Wright inequalities for
sparse quadratic forms in subgaussian random variables.
This provides useful concentration inequalities for sparse subgaussian random
vectors in two ways.
Let $X = (X_1, \ldots, X_m) \in \R^m$ be a random vector with
independent subgaussian components, and $\xi =(\xi_1,
\ldots, \xi_m) \in \{0, 1\}^m$ be independent Bernoulli random variables.
We prove the large deviation bound for a sparse quadratic form of 
$(X \circ \xi)^T A (X \circ \xi)$, where $A \in \R^{m \times m}$ is an
$m \times m$ matrix, and random vector $X \circ \xi$ denotes the
Hadamard  product of an isotropic subgaussian random vector $X \in \R^m$
and a random vector $\xi  \in \{0, 1\}^m$ 
such that $(X \circ \xi)_{i} = X_{i}  \xi_i$, where $\xi_1, \ldots,\xi_m$ are independent Bernoulli 
random variables.  The second type of sparsity in a quadratic form comes from the setting
where we randomly sample the elements of an anisotropic subgaussian vector $Y =
H X$ where $H \in \R^{m\times  m}$ is an $m \times m$ symmetric matrix;
we study the large deviation bound on the $\ell_2$-norm
$\twonorm{D_{\xi} Y}^2$ from its expected value, where for a given
vector $x \in \R^m$,  $D_{x} = \diag(x)$ denotes the diagonal matrix whose main diagonal entries are the
entries of $x$. This form arises naturally from the context of
covariance estimation.
\end{abstract}

{\bf Keywords:} Hanson-Wright inequality; Subgaussian concentration;
Sparse quadratic forms. \\
 
AMS 2010 Subject Classification: 60B20.

\input{main}

\input{shortsparse}

\input{sparseproof}

\input{noncenter}

\section*{Acknowledgements}
Mark Rudelson encouraged me to apply the method from~\cite{RV13} to
prove the first result in the current paper.
The author is also grateful for discussions with Tailen Hsing, 
which helped improving the presentation of this paper tremendously.
This work was supported in part by NSF under Grant DMS-1316731 and
Elizabeth Caroline Crosby Research Award from the Advance
Program at the University of Michigan. The proof presented here was
filed as Technical Report, 539, October, 2015, Department of Statistics, University of Michigan.

\appendix

\section{Proof of Lemma~\ref{lemma::expmgf2} }
\begin{proofof2}
Let  $Z := X^2 - \E X^2$ and 
$Y :={Z}/{\norm{Z}_{\psi_1}}$.
Then  $Y$ and $Z$ are both centered sub-exponential random variables with
$\norm{Y}_{\psi_1} =1$ and 
\bens
\norm{Z}_{\psi_1} = \norm{X^2 -\E X^2}_{\psi_1}
 \le 2 \norm{X^2}_{\psi_1} \le 4\norm{X}^2_{\psi_2} \le 4K^2
\eens
which follows from the triangle inequality and Lemma 5.14 of~\cite{Vers12}.

Now set $t :=\tau \norm{X^2 -\E X^2}_{\psi_1}$, where
for $\abs{\tau} \le \inv{23.5 e K^2}$,
\bens
e \abs{t} = e \abs{\tau} \norm{X^2 -\E X^2}_{\psi_1}
\le \frac{ 4 K^2 }{23.5 K^2} < \frac{8}{47} 
\eens
and $2(e \abs{t})^3 \le (e \abs{t})^2$.
By Lemma 5.15 of~\cite {Vers12}, we have for all $k$,
\bens
\E \exp(t Y)
& = & 
1 + t \E Y + \sum_{p=2}^{\infty} \frac{t^p \E Y^p}{p!} 
\le 1 + \sum_{p=2}^{\infty} \frac{\abs{t}^p \E \abs{Y}^p}{p!} \\
& \le & 
1 + \sum_{p=2}^{\infty} \frac{\abs{t}^p p^p}{p!} 
\le 1 + \sum_{p=2}^{\infty} \frac{\abs{t}^p e^p}{\sqrt{2\pi p}} 
\eens
Thus 
\bens
\E \exp(t Y)
& \le &
1 +  \frac{(t e)^2}{2\sqrt{\pi}}
+ \inv{\sqrt{6\pi}}\sum_{p=3}^{\infty} (e \abs{t})^p  \\
& \le &  
1+  \frac{\abs{t}^2 e^2}{2\sqrt{\pi}}
+ \inv{\sqrt{6\pi}}\frac{ (e \abs{t})^3}{1- e\abs{t}}  
\le 1+  \frac{\abs{t}^2 e^2}{2\sqrt{\pi}}
+ \frac{8(e \abs{t})^2}{39\sqrt{6\pi}}\\
& < &  
1+ e^2\tau^2 \norm{X^2 -\E X^2}_{\psi_1}^2
(\inv{2\sqrt{\pi}} +  \frac{8}{39\sqrt{6 \pi}}) \\
& \le & 
1+ 38.94\abs{\tau}^2  K^4
\eens
where we used the following form of Stirling's approximation
for all $p \ge 2$,
\bens
\inv{p!} \le \frac{e^p}{p^p} \inv{\sqrt{2 \pi p}}
 \le \frac{e^p}{p^p} \inv{2\sqrt{\pi}}.
\eens
The lemma is thus proved given that
\bens
 \E \exp \tau (X^2 - \E X^2) =
 \E \exp\left(\tau \norm{X^2 - \E X^2}_{\psi_1} Y \right)
 = \E \exp(t Y).
\eens
\end{proofof2}

\section{Proof of Lemma~\ref{lemma::taylor}}
\label{sec::Taylor}
\begin{proofof2}
Note that the following
holds by Lemma 5.14~\cite {Vers12},
\bens
\norm{X_i}^2_{\psi_2} \le \norm{X_i^2}_{\psi_1} \le 2
\norm{X_i}^2_{\psi_2} = 2K^2.
\eens
For all $k$, let  $Y_k :={X_k^2}/{\norm{X_k^2}_{\psi_1} }$. 
By definition, $Y_k$ is a sub-exponential random
variable with $\norm{Y_k}_{\psi_1} = 1$.
We now set
$t_k :=\lambda a_{kk} \norm{X_k^2}_{\psi_1}$.
Following the  proof of Lemma~\ref{lemma::expmgf2}, we
first use the Taylor expansions to obtain for all $k$,
\bens
\lefteqn{
\E \exp(t_k Y_k)  := \E \exp\left(\lambda a_{kk}X_k^2\right) = 
1 + t_k \E Y_k + \sum_{p=2}^{\infty} \frac{t_k^p \E Y_k^p}{p!} } \\
& \le &  1 + t_k \E Y_k + \sum_{p=2}^{\infty} \frac{\abs{t_k}^p \E \abs{Y_k}^p}{p!} \\
& \le &  
1+ \lambda a_{kk}  \E X_k^2  + \frac{\abs{t_k}^2 e^2}{2\sqrt{\pi}}
+ \inv{\sqrt{6\pi}}\sum_{p=3}^{\infty} (e \abs{t_k})^p  \\
& \le &  
1+ \lambda a_{kk} \E X_k^2 +
 \frac{\abs{t_k}^2 e^2}{2\sqrt{\pi}}
+ \inv{\sqrt{6\pi}}2 (e \abs{t_k})^3  \\
& < &  
1+ \lambda a_{kk} \E X_k^2 +  e^2
(\lambda a_{kk} \norm{X_k^2}_{\psi_1} )^2 (\inv{2\sqrt{\pi}} +
\inv{\sqrt{6 \pi}})  \le  
1+ \lambda a_{kk} \E X_k^2 
+ 16 \abs{\lambda a_{kk}}^2  K^4.
\eens
where
$$e \abs{t_k} \le 
\frac{ \abs{a_{kk}}\norm{X_k^2}_{\psi_1}  }{4 K^2 \max_{k}
  \abs{a_{kk}}} \le \inv{2}$$ and $2(e \abs{t_k})^3 \le (e \abs{t_k})^2$. 
The lemma is thus proved.
\end{proofof2}

\input{gaussian}

\bibliography{subgaussian}

\end{document}

%% file: main.tex
\section{Introduction}
\label{sec::intro}
In this paper, we explore the concentration of measure results for
quadratic forms involving a sparse subgaussian random vector $X \in \R^m$.
Sparsity can naturally come from the fact that the high dimensional
vector $X \in \R^m$ is sparse, for example, when the elements of $X$
are missing at random, or when we intentionally sparsify the vector
$X$ to speed up computation. The purpose of the paper is to prove the
Hanson-Wright type of large deviation bounds for sparse quadratic
forms in Theorems~\ref{thm::HWsparse} and \ref{thm::MNsparse}.

Sparsity comes in two forms. In Theorem~\ref{thm::HWsparse}, we
randomly sparsify the subgaussian vector $X$ involved in the quadratic
form $X^T A X$, where $X = (X_1, \ldots, X_m) \in \R^m$ is a random vector with
independent subgaussian components, and $\xi =(\xi_1, \ldots, \xi_m)
\in \{0, 1\}^m$ consists of independent Bernoulli random variables. 
In particular, we first consider  $(X \circ \xi)^T A(X \circ \xi)$,  
where $X \circ \xi \in \R^m$ denotes the Hadamard  product of random
vectors $X$ and $\xi$ such that $(X \circ \xi)_{i} = X_{i}  \xi_i$ and $A$ is an
$m \times m$ matrix.
The second type of sparsity comes into play when we sample the
elements of an anisotropic subgaussian random vector $Y = D_0 X$ where $X \in
\R^m$ is as defined in Theorem~\ref{thm::HWsparse} and $D_0 \in
\R^{m\times m}$ is an $m \times m$ symmetric matrix.
 
The bound in Theorem~\ref{thm::MNsparse} 
allows the second type of sparsity in a quadratic form in the following sense.
Suppose $A_0$ is an $m \times m$ symmetric positive semidefinite matrix and
$A_0^{1/2}$ is the unique square root of $A_0$. Suppose we randomly sample the rows or
columns of $A_0^{1/2}$ to construct a quadratic form as follows,
\ben
\label{eq::sparse2}
X^T A_0^{1/2}A_0^{1/2} X \to X^T A_0^{1/2} D_{\xi} A_0^{1/2} X.
\een
We state in Theorem~\ref{thm::MNsparse}, where we replace $A_0^{1/2}$
with $D_0$, a symmetric $m \times m$ matrix, 
the large deviation bound for the sparse quadratic form on the right hand side of \eqref{eq::sparse2}. 
These questions arise naturally in the context of covariance estimation problems, where we naturally take $A_0$ and $D_0$ as symmetric positive (semi)definite matrices.

The following definitions correspond to Definitions 5.7 and 5.13
in~\cite{Vers12}.
For a random variable $X$, the sub-gaussian (or $\psi_2$)
 norm of $X$ denoted by $\norm{X}_{\psi_2}$ is defined to be~\cite{Vers12}:
\bens
\label{eq::defineK}&&
\norm{X}_{\psi_2} = \sup_{p \ge 1} p^{-1/2}
(\mathbb{E}\abs{X}^p)^{1/p} \; \;
\text{which  is the smallest $K_2$} \\
&&
\; \; \; \; \; \text{ which satisfies } \;  
\left(\mathbb{E}\abs{X}^p\right)^{1/p }\le K_2\sqrt{p} \;\;\;  \forall p
\ge 1;  \\
&&\label{eq::defineC}
\text{if} \; \; \E[X]=0, \; \text{then} \; \; \E\exp(tX) \le \exp(C
t^2 \norm{X}_{\psi_2}^2 ) \; \; \text{for all } \; \; t \in \R.
\eens

For a symmetric matrix $A = (a_{ij}) \in \R^{m \times m}$, we let 
$\lambda_{\max}(A)$ and $\lambda_{\min}(A)$ denote the largest and the
smallest eigenvalue of $A$ respectively. 
Moreover, we order the $m$ eigenvalues algebraically and denote them by
\bens
\lambda_{\min}(A) = \lambda_1(A) \le \lambda_2(A) \le \ldots \le
\lambda_m(A) =\lambda_{\max}(A).
\eens
For a matrix $A$, the operator norm $\twonorm{A}$
is defined to be $\sqrt{\lambda_{\max}(A^T A)}$. Let $\fnorm{A} =
(\sum_{i, j} a_{ij}^2)^{1/2}$. Let $\diag(A)$ be the diagonal of $A$.  Let $\offd(A)$ be the off-diagonal of $A$. 
For matrix $A$, $r(A)$ denotes the effective rank  ${\tr(A)}/{\twonorm{A}}$.
For two numbers $a, b$, $a \vee b := \max(a, b)$.

In particular, we prove:
\begin{theorem}
\label{thm::HWsparse}
Let $X = (X_1, \ldots, X_m) 
\in \R^m$ be a random vector with independent components
 $X_i$ which satisfy $\expct{X_i} = 0$ and $\norm{X_i}_{\psi_2} \leq K$. 
Let $\xi = (\xi_1, \ldots, \xi_m) \in \{0, 1\}^m$ be a random vector
independent of $X$, with independent Bernoulli random variables $\xi_i$ such that
$\E(\xi_i) = p_i$. Let $A =(a_{ij})$ be an $m \times m$ matrix. Then, for every $t > 0$,
\ben
\nonumber
\lefteqn{
\prob{\abs{(X \circ \xi)^T A (X \circ \xi)
 - \expct{(X \circ \xi)^T A (X \circ \xi)}
} > t} \leq } \\
\label{eq::q1}
& &  2\exp \left(- c\min\left(\frac{t^2}{K^4 \left(\sum_{k=1}^m p_k
      a_{kk}^2 + \sum_{i\not=j} a_{ij}^2 p_i p_j \right)
},    \frac{t}{K^2 \twonorm{A}} \right)\right) 
\een
where $X \circ \xi$ denotes the Hadamard  product of random vectors
$X$ and $\xi$ such that $(X \circ \xi)_{i} = X_{i}  \xi_i$.
\end{theorem}
Let $\xi$ be as defined in Theorem~\ref{thm::HWsparse}.
We now randomly sample entries of a correlated subgaussian random vector $Y = D_0 X$
and study the large deviation bound on the norm of $\twonorm{D_{\xi}
  Y}^2$ from its expected value in Theorem~\ref{thm::MNsparse}, where
for a given $x \in \R^m$,  $D_{x} = \diag(x)$ denotes the diagonal matrix whose main diagonal entries are
the elements of $x$. And we write $D_x := \diag(x)$ interchangeably.
Partition a symmetric matrix $D_0 \in \R^{m \times m}$ according to its columns as
$D_0 = [d_1, d_2, \ldots, d_m]$. Denote by
\ben
\label{eq::eigendec}
A_0 & := & D_0^2 = \sum_{i=1}^m d_i d_i^T = (a_{ij}) \succeq 0.
\een
The bounds in Theorem~\ref{thm::HWsparse} and
Theorem~\ref{thm::MNsparse} reduce to essentially the same type.
\begin{theorem}
\label{thm::MNsparse}
Let $D_{\xi}$ be a diagonal matrix with elements from the random
vector $\xi \in \{0, 1\}^m$, where $\E \xi_j = p_j$, for $0 \le p_j
\le 1$.  Let $X$ be as defined in Theorem~\ref{thm::HWsparse}, independent of $\xi$.
Let $A_0 = (a_{ij}) =D_0^2$.  
Let $Y = D_0 X$. Then, for every $t > 0$,
\bens
\label{eq::sparseYoffd}
\lefteqn{\prob{\abs{Y^T D_{\xi} Y - \E Y^T D_{\xi} Y} > t}  =: \prob{\abs{S}>t}} \\
& \le & 
\nonumber
2\exp \left(-c_2\min\left(\frac{t^2}{K^4 (\sum_{i=1}^m p_i a^2_{ii} + \sum_{i\not=j} a^2_{ij} p_i p_j)},
    \frac{t}{K^2 \twonorm{A_0}} \right)\right)
\eens
where $c_2, C$ are some absolute constants.
\end{theorem}
Theorem~\ref{thm::Bernmgf} 
shows a concentration of measure bound on a
quadratic form with Bernoulli random variables where an explicit
dependency on $p_i$, for all $i$, is shown. The setting here is
different from Theorem~\ref{thm::HW} as we deal with a quadratic form
which involves non-centered Bernoulli random variables.
Theorem~\ref{thm::Bernmgf} is crucial in proving
Theorem~\ref{thm::MNsparse}.
\begin{theorem}
\label{thm::Bernmgf}
Let $\xi = (\xi_1, \ldots, \xi_m) \in \{0, 1\}^m$ be a random vector
with independent Bernoulli random variables $\xi_i$ such that
$\xi_i = 1$ with probability $p_i$ and $0$ otherwise.
Let $A = (a_{ij})$ be an $m \times m$ matrix. 
Then, for every $0 \le \lambda \le \inv{104 \max(\norm{A}_1, \norm{A}_{\infty})}$,
\bens
\lefteqn{
\E \exp\left(\lambda \sum_{i,j} a_{ij}  \xi_i \xi_j \right)
\le
\exp\left(\lambda
\left(\sum_{i=1}^m a_{ii} p_i 
+\sum_{i\not= j} a_{ij}  p_i p_j \right)\right)
*} \\
&& \exp\left(\frac{1}{3}\lambda\sum_{j \not=i}\abs{a_{ij}} \sigma^2_i 
\sigma^2_j \right) * 
\exp\left(C_5 \lambda \left(\half\sum_{i=1}^m \abs{a_{ii}} p_i  +\sum_{j\not=i}  \abs{a_{ij}} p_jp_j 
\right)  \right)
\eens
where $\sigma_i^2 = p_i(1-p_i)$ and $C_5 \le 0.04$.
\end{theorem}
The proof of Theorem~\ref{thm::Bernmgf}  is deferred to Section~\ref{sec::noncenter}.

Before we leave this section, we also introduce the following notation.
For a random variable $X$, the sub-exponential  (or $\psi_1$) norm of $X$ denoted by 
$\norm{X}_{\psi_1}$ is defined to be  the smallest $K_2$
 which satisfies 
\bens
&& 
\nonumber
\left(\mathbb{E}\abs{X}^p\right)^{1/p }\le K_2 p \;  \forall p
\ge 1; \; \; \text{in other words} \; \; \\
\label{eq::expK}
&& \norm{X}_{\psi_1} = \sup_{p \ge 1} p^{-1}
(\mathbb{E}\abs{X}^p)^{1/p}.
\eens
Throughout this paper $C_0, C, C_1, c, c_1, \ldots$ denote positive absolute constants whose value may change from line to line.
For a vector $X \in \R^m$, let 
$X_{\Lambda_{\delta}}$ denote $(X_i)_{i \in \Lambda_{\delta}}$ for
a set $\Lambda_{\delta} \subseteq [m]$.

We use the following properties of the Hadamard product~\cite{HJ91},
\bens
A \circ x x^T & = &D_x A D_x \\
\text { and } 
\tr(D_{\xi} A D_{\xi} A^T) & = &
 \xi^T(A \circ A) \xi 
 \eens
from which a simple consequence is
$\tr(D_{\xi} A D_{\xi}) = \xi^T(A \circ I) \xi  =  \xi^T \diag(A)) \xi$.

We use the following bounds throughout our paper.
For any $x \in \R$, 
\ben 
\label{eq::elem}
e^x \le 1 + x + \half x^2 e^{\abs{x}}.
\een
We need the following result which follows from Proposition 3.4 in~\cite{Tal95}.
\begin{lemma}
\label{lemma::bern-sum}
Let $A = (a_{ij})$ be an $m \times m$ matrix.
Let $a_{\infty} := \max_i \abs{a_{ii}}$.
Let $\xi = (\xi_1, \ldots, \xi_m) \in \{0, 1\}^m$ be a random vector
with independent Bernoulli random variables $\xi_i$ 
such that
$\xi_i = 1$ with probability $p_i$ and $0$ otherwise.
Then for $\abs{\lambda} \le \inv{4 a_{\infty}}$,
\bens
\E\exp\left(\lambda \sum_{i=1}^m a_{ii} (\xi_i -p_i)\right) 
 & \le & 
\exp\left(\inv{2} \lambda^2 e^{\abs{\lambda} a_{\infty}} \sum_{i=1}^m
  a_{ii}^2 \sigma^2_i \right)
\eens
where
$\sigma^2_i = p_i(1-p_i)$.
\end{lemma}
We need to state Lemma~\ref{lemma::expmgf2}, which provides an
estimate of the moment generating function  for the centered
sub-exponential random variable $Z_k := X_k^2 - \E X_k^2$ for $X_k$ as
defined in Theorem~\ref{thm::HWsparse}.
\begin{lemma}
\label{lemma::expmgf2}
Let $X \in \R$ be a sub-gaussian random variable which satisfies 
$\expct{X} = 0$ and $\norm{X}_{\psi_2} \leq K$.  
Let $\abs{\tau} \le \inv{23.5 e K^2}$. Denote by $C_0 :=38.94$. Then
\bens
\E\left(\exp(\tau (X^2 - \E X^2))\right) \le 1+ 38.94 
\tau^2 K^4 \le \exp(C_0 \tau^2 K^4).
\eens
\end{lemma}
The proof follows essentially that of Lemma 5.15
in~\cite{Vers12}; we provide  here explicit constants.

The rest of the paper is organized as follows.
In Section~\ref{sec::related}, we compare our results with those in the
literature. We then prove Theorem~\ref{thm::HWsparse} in Section~\ref{sec::proofofHW} and 
Theorem~\ref{thm::MNsparse} in Section~\ref{sec::proofofMN}. We prove 
Theorem~\ref{thm::Bernmgf}  in Section~\ref{sec::noncenter}. We leave
certain calculations in Appendix A for the purpose of 
self-containment, namely, the proof of Lemmas~\ref{lemma::expmgf2} and~\ref{lemma::taylor}.

\section{Consequences and related work}
\label{sec::related}
In this section, we first compare with the following form of the Hanson-Wright
inequality as recently derived in~\cite{RV13}, as well as an even more closely related result in~\cite{Rud15}.
Such concentration of measure bounds were originally proved
by~\cite{HW71,HW73}. 
The bound as stated in Theorem~\ref{thm::HW} is proved in~\cite{RV13}.
\begin{theorem}{\cite{RV13}}
\label{thm::HW}
Let $X = (X_1, \ldots, X_m) \in \R^m$ be a random vector with independent components $X_i$ which satisfy
$\expct{X_i} = 0$ and $\norm{X_i}_{\psi_2} \leq K$. Let $A$ be an $m \times m$ matrix. Then, for every $t > 0$,
\bens
\prob{\abs{X^T A X - \expct{X^T A X} } > t} 
\leq 
2 \exp \left(- c\min\left(\frac{t^2}{K^4 \fnorm{A}^2}, \frac{t}{K^2 \twonorm{A}} \right)\right).
\eens
\end{theorem}
When $X$ is a vector whose coordinates are $\pm 1$
Bernoulli random variables, the following Lemma in the same spirit as in
Theorem~\ref{thm::HWsparse} is shown in~\cite{Rud15}.
\begin{lemma}{(\cite{Rud15})}
\label{lemma::randomset}
Let $J$ be a random subset of $[m]$ of size $k < m$ uniformly chosen 
among all such subsets. Denote by $R_J = \sum_{j \in J} e_j e_j^T$ the coordinate projection
on the set $J$.
Let $Y = (\ve_1, \ldots, \ve_m)$ be  vector whose coordinates are $\pm
1$ Bernoulli Random variables. Then for any $m \times m$ matrix $A$
and any $t > 0$
\bens
\lefteqn{
\prob{\abs{Y^T R_J A R_J Y  - \expct{Y^T R_J A R_J Y}
}> t} } \\
& \le & 
2 \exp \left(- c\min\left(\frac{t^2}{k \twonorm{A}^2}, 
    \frac{t}{\twonorm{A}} \right)\right).
\eens
\end{lemma}
Other related results include~\cite{Lat06,HKZ12,DKN10,BM12,FR13,AW15}.
We refer to~\cite{RV13} for a survey of these and other related
results.

Clearly, the large deviation bounds in
Theorems~\ref{thm::HWsparse} and \ref{thm::MNsparse}
are determined by the following quantity 
\bens
\bar{M} :=  \sum_{i=1}^m p_i a^2_{ii} + \sum_{i\not=j} a^2_{ij} p_i
p_j
\eens
We now state some consequences of Theorems~\ref{thm::HWsparse}
and~\ref{thm::MNsparse} in Corollaries~\ref{coro::HWsparse} and
\ref{coro::MN-sparse}.  
\silent{For simplicity, we first take $p_1 = \ldots =
p_m = p$. Then 
\ben
\label{eq::Mbound}
\bar{M} &:= & \sum_{i=1}^m p_i a^2_{ii} + \sum_{i\not=j} a^2_{ij}
p_i p_j =p \fnorm{\diag(A_0)}^2 + p^2 \fnorm{\offd(A_0)}^2.
\een}
Lemma~\ref{lemma::randomset} and  Corollaries~\ref{coro::HWsparse} and
~\ref{coro::MN-sparse} show essentially a large deviation bound at roughly the same
order given that 
$$p \fnorm{\diag(A)}^2 + p^2 \fnorm{\offd(A)}^2 \le p m
\twonorm{A}^2$$ while 
$k \twonorm{A}^2 = \frac{k}{m} m \twonorm{A}^2$.

The following Corollary~\ref{coro::HWsparse} follows from
Theorem~\ref{thm::HWsparse} immediately.
\begin{corollary}
\label{coro::HWsparse}
Let $X, \xi$ be as defined in Theorem~\ref{thm::HWsparse}.
Let $p_1=p_2 = \ldots = p_m =p$. Let $A = (a_{ij})$ be an $m \times m$ matrix. Then, for every $t > 0$,
\bens
\lefteqn{
\prob{\abs{X^T D_{\xi} A D_{\xi} X  - \expct{X^T D_{\xi} A D_{\xi} X}} > t} \leq } \\
& & 
2\exp \left(- c \min\left(\frac{t^2}{K^4 \left(p\fnorm{\diag(A)}^2
+ p^2 \fnorm{\offd(A)}^2\right)},\frac{t}{K^2 \twonorm{A}} \right)\right).
\eens
\end{corollary}
\begin{corollary}
\label{coro::MN-sparse}
Let $D_0, A_0, X, \xi, Y$ be as defined in
Theorem~\ref{thm::MNsparse}.
Let $p_1=p_2 = \ldots = p_m =p$. Then, for every $t > 0$,
\bens
\lefteqn{
\prob{\abs{Y^T D_{\xi} Y - \E Y^T D_{\xi} Y} > t} = 
\prob{\abs{\twonorm{D_{\xi} D_0 X }^2 - \expct{\twonorm{D_{\xi} D_0 X}}^2} > t} } \\
& \le & 
2\exp \left(-c\min\left(\frac{t^2}{K^4 (p \fnorm{\diag(A_0)}^2 +p^2 \fnorm{\offd(A_0)}^2)}, 
\frac{t}{K^2 \twonorm{A_0}} \right)\right).
\eens
\end{corollary}

\begin{corollary}
\label{coro::HW2}
Suppose all conditions in Corollary~\ref{coro::HWsparse} hold.
Let $A \in \R^{m \times m}$ be positive semidefinite.
Suppose $EX_i^2 =1$ and
\ben
\label{eq::ratecond}
\log m \twonorm{A} = o(p \tr(A))
\een
Then with probability at least $1-4/m^4$,
\bens
 \abs{X^T D_{\xi} A D_\xi  X} 
& \le  & p\tr(A) (1+o(1)).
\eens
\end{corollary}

\begin{proof}
Define 
\bens
S = \sum_{i,j} a_{ij}(X_i \xi_i X_j \xi_j - \expct{X_i \xi_i X_j\xi_j}).
\eens
Thus $\E S = \sum_{i} a_{ii} \expct{X_i^2}\expct{\xi_i} =  p \tr(A)$.
We have under conditions of Theorem~\ref{thm::HWsparse}, with
probability at least $1-4/m^4$,
for some absolute constant $C$,
\bens
\lefteqn{
\abs{S} := \abs{X^T D_{\xi} A D_{\xi}  X - p\tr(A)}} \\
& \le &
C K^2 \log^{1/2} m \left( \log^{1/2} m \twonorm{A} + \sqrt{p}
   \fnorm{\diag(A)} +  p \fnorm{\offd(A)} \right) =: t
\eens
where under condition \eqref{eq::ratecond}, the deviation term is of a small order of the
expected value $p \tr(A)$; that is,
\bens
t \asymp \log m \twonorm{A} + \log^{1/2} m (\sqrt{p} \fnorm{\diag(A)} +  p
  \fnorm{A} )  =: I + II = o(p\tr(A)).
\eens
To see this, notice that \eqref{eq::ratecond} immediately implies that 
the first term in $t$ is of $o(p\tr(A))$.
Now in order for the second and third term to be of $o(p \tr(A))$,
we need that 
\bens
\sqrt{p} \fnorm{A} \log^{1/2} m  & \ll & p\tr(A) \text{ and hence }  p \gg \log m \fnorm{A}^2/\tr(A)^2
\eens
which is satisfied by \eqref{eq::ratecond} given that
$\frac{\twonorm{A}}{\tr(A)} \ge \frac{\fnorm{A}^2}{\tr(A)^2}$, which
in turn is due to $\fnorm{A}^2 \le
\tr(A) \twonorm{A}$.
\end{proof}

\begin{corollary}
\label{coro::final}
Suppose that \eqref{eq::ratecond} and all conditions in Corollary~\ref{coro::MN-sparse} hold.
Assume $EX_i^2 =1$.
Then with probability at least $1-\frac{4}{m^4}$, $ \abs{X^T D_0 D_{\xi} D_0  X} = p\tr(A_0) (1+o(1)).$
\end{corollary}

\begin{proof}
First by independence of $X$ and $\xi$, we have for $\E X_i^2 = 1$,
\bens
\expct{X ^T A_{\xi} X} & = & \E \sum_{k=1}^m X_k^2  A_{\xi, kk}=
\sum_{k=1}^m \E (X_k^2) \E (A_{\xi, kk}) \\
& = & 
\sum_{k=1}^m \E X_k^2 \E \sum_{\ell=1}^m \xi_{\ell}  d_{k \ell}^2 
= \sum_{\ell=1}^m p_{\ell}  \sum_{k=1}^m   d_{k \ell}^2 = 
\sum_{\ell=1}^m p_{\ell}  a_{\ell \ell}.
\eens
We have by Corollary~\ref{coro::MN-sparse}, with probability at least $1-\frac{4}{m^4}$,
\bens
 \abs{X^T D_0 D_{\xi} D_0  X}
& \le  &
  \sum_{i=1}^m a_{ii} p_i + C K^2 \log^{1/2} m \left( \log^{1/2} m \twonorm{A_0} +   
\sqrt{\bar{M}} \right) \\
& \le  &  
p\fnorm{D_0}^2 + C K^2 \log^{1/2} m \left( \log^{1/2} m \twonorm{A_0} + \sqrt{p} \fnorm{\diag(A_0)} +  p \fnorm{\offd(A_0)} \right).
\eens 
for some absolute constants $C$, where 
$\sqrt{\bar{M}} \le\sqrt{p} \fnorm{\diag(A_0)} +  p \fnorm{\offd(A_0)}$.
The rest of the proof for Corollary~\ref{coro::final} follows from that of 
Corollary~\ref{coro::HW2}.
\end{proof}

\subsection{Implications when $p_1, \ldots, p_m$ are not the same}
\label{sec::prelim}
We first need the following sharp statements about eigenvalues of a 
Hadamard product.  See for example Theorem~5.3.4~\cite{HJ91}.
\begin{theorem}
\label{eq::tightHard}
Let $A, B \in \R^{m \times m}$ be positive semidefinite.
Let $a_{\infty} := \max_{i=1}^m a_{ii}$ and $b_{\infty} := \max_{i=1}^m b_{ii}$.
Any eigenvalue of $\lambda(A \circ B)$ satisfies
\bens
\lambda_{\min}(A)\lambda_{\min}(B)
& \le & 
(\min_{i=1}^m a_{ii})\lambda_{\min}(B) \\
& \le & 
\lambda(A \circ B) \\
& \le & a_{\infty} \lambda_{\max}(B)
\le \lambda_{\max}(A)\lambda_{\max}(B).
\eens
\end{theorem}

\begin{corollary}
\label{coro::tail}
Suppose all conditions in Theorem~\ref{thm::MNsparse} hold.
Suppose $EX_i^2 =1$.  Let $\vecp = (p_1, \ldots, p_m)$.
Let $\onenorm{\vecp} := \sum_{i=1}^m p_i$ and $\twonorm{\vecp}^2 =
\sum_{i=1}^m p_i^2$. 
Then with probability at lest $1-4/m^4$,
\bens
 \abs{X^T D_0 D_{\xi} D_0  X}
& \le  & 
 \sum_{i=1}^m  p_i \twonorm{d_i}^2 + C K^2 \log^{1/2} m  \twonorm{D_0}  \left( \log^{1/2} m
   \twonorm{D_0} + 2(\max_{i} \twonorm{d_i} )\onenorm{\vecp}^{1/2}  \right).
\eens
\end{corollary}

\begin{proof}
Recall $A_0 = (a_{ij}) = D_0^2 \succeq 0$.
Let $a_{\infty} := \max_{i=1}^m a_{ii} =\max_{i} \twonorm{d_i}^2$. 
Thus we have $a_{\infty} \le \twonorm{D_0}^2$.
Denote by $p = (p_1, \ldots,p_m)$. We have by Theorem~\ref{eq::tightHard},
\bens
\bar{M} & = & \sum_{i=1}^m p_i a^2_{ii} + \sum_{i\not=j} a^2_{ij} p_i p_j 
\le \sum_{i=1}^m p_i a^2_{ii} + \vecp^T( A_0 \circ A_0) \vecp \\
& \le & 
 a_{\infty}^2 \onenorm{\vecp}    
+ \lambda_{\max}(A_0 \circ A_0)\twonorm{\vecp}^2 \\
& \le &   a_{\infty}^2 \onenorm{\vecp}  
+  a_{\infty} \twonorm{A_0} \twonorm{\vecp}^2
 \le 2 a_{\infty} \twonorm{A_0} \onenorm{\vecp}.
\eens
where  $\twonorm{\vecp}^2 \le \onenorm{\vecp}$.
The corollary thus follows immediately from Theorem~\ref{thm::MNsparse}.
\end{proof}

\begin{remark}
Assume that $p_i \ge \frac{\log m}{m}$ and hence $\onenorm{\vecp} \ge \log
m$. Then we have $\twonorm{\vecp} \le \onenorm{\vecp}^{1/2} \le \onenorm{\vecp}$.
Notice that the second term starts to dominate when $\onenorm{\vecp}^{1/2}
\gg \log m$ while the total deviation remains to be a small order of the
mean $\sum_{i=1}^m p_i \twonorm{d_{i}}^2$ so long as 
$$\onenorm{\vecp} \gg
\frac{\max_k \twonorm{d_k}^2 \twonorm{D_0}^2}{\min_{k} \twonorm{d_k}^4}.$$
\end{remark}

%% file: shortsparse.tex
\section{Proof of Theorem~\ref{thm::HWsparse}}
\label{sec::proofofHW}
\begin{proofof2}
The structure of our proof follows that of Theorem 1.1
by~\cite{RV13}. The problem reduces to estimating the diagonal and the
off-diagonal sums.

\noindent{\bf Part I: Diagonal Sum.}
Define 
\ben
\label{eq::defineS0}
S_0 := 
\sum_{k=1}^m   a_{kk} \xi_k X_k^2-
\E\sum_{k=1}^m   a_{kk} \xi_k  X_k^2 \; \; \text{ where } \; 
\E\sum_{k=1}^m  a_{kk} \xi_k X_k^2 = \sum_{k=1}^m  a_{kk} p_k
 \E X_k^2.
\een
\begin{lemma}
\label{lemma::exp}
Let $X$ and $\xi$ be defined as in Theorem~\ref{thm::HWsparse}.
Let $A$ be an $m \times m$ matrix. Then, for every $t > 0$,
\bens
\lefteqn{
\prob{\abs{
\sum_{k=1}^m a_{kk} \xi_k X_k^2 - \sum_{k=1}^m a_{kk} p_k \E X_k^2}  >
t} \le \prob{S_0  > t} + \prob{S_0  < - t}} \\
& \le & 
2\exp \left[- \inv{4e}\min\left(\frac{t^2}{3 K^4 \sum_{k=1}^m a_{kk}^2
    p_k},
\frac{t}{K^2 \max_{k} \abs{a_{kk}}} \right)\right].
\eens
\end{lemma}
We prove Lemma~\ref{lemma::exp} after we state Lemma~\ref{lemma::taylor}.
For the general case where $X_k$ are mean-zero independent sub-gaussian random
variables with $\norm{X_k}_{\psi_2} \le K$,
we first state the following bound on the moment generating function
of $X_k^2$. 
\begin{lemma}
\label{lemma::taylor}
Suppose that $\abs{\lambda} < 1/(4  e K^2 \max_{k}
\abs{a_{kk}})$.
Then for all $k$, we have for all $a_{kk} \in \R$
\ben
\label{eq::expmgf}  
\E \exp\left(\lambda a_{kk}  X_k^2\right)  -1 
 \le \lambda a_{kk} \E X_k^2 + 16
\lambda^2 a_{kk}^2  K^4.
\een
\end{lemma}

\begin{proofof}{Lemma~\ref{lemma::exp}}
We first state some simple fact: $\max_{k=1}^m \E X_i^2 \le   K^2.$
By independence of $X_1, \ldots, X_k$
and $\xi_1, \ldots, \xi_k$, 
we  bound the moment generating function of $S_0$ as follows: 
for $\abs{\lambda} \le \inv{4 e  K^2\max_{k} \abs{a_{kk}}}$
\bens
\lefteqn{
\E \exp(\lambda S_0) = 
\E \exp(\lambda  \sum_{k=1}^m  \xi_k X_k^2 a_{kk} - 
\lambda\sum_{k=1}^m  p_k a_{kk} \E X_k^2) }\\
& = &
\prod_{k=1}^m \left(\frac{\E \exp(\lambda   a_{kk} \xi_k X_k^2)}
{\exp(\lambda p_k a_{kk} \E X_k^2)}\right)  = 
\prod_{k=1}^m 
\frac{\E_{\xi}\E_{X} 
 \exp(\lambda   a_{kk} \xi_k X_k^2)}{\exp(\lambda
  p_k a_{kk} \E X_k^2)} \\
& \le &
\label{eq::chimgf3}
\prod_{k=1}^m \frac{1+p_k \left(\lambda a_{kk} \E X_k^2 + 16
\lambda^2 a_{kk}^2 K^4\right)}{\exp(\lambda p_k a_{kk} \E X_k^2)} \\
 & \le &
\prod_{k=1}^m \frac{
\exp\left(\lambda p_k a_{kk} \E X_k^2 + 16
\lambda^2 p_k a_{kk}^2  K^4\right)}
{\exp(\lambda p_k a_{kk} \E X_k^2)}  = 
\exp\left(16\lambda^2 K^4 \sum_{k=1}^m 
p_k a_{kk}^2\right)
\eens
where we used \eqref{eq::expmgf} for the first inequality and the fact
that $1+ x\le e^x$ for the second inequality.
Hence for $0< \lambda \le \inv{4 e K^2\max_{k}\abs{a_{kk}}}$, we have 
\bens
\prob{S_0 > t} \le 
\frac{\E \exp(\lambda S_0)}{e^{\lambda t}} \le  
\exp\left(-\lambda t + 
16 K^4 \lambda^2 \sum_{k=1}^m p_k  a_{kk}^2 \right) 
\eens
for which the optimal choice of $\lambda$ is
\bens
\label{eq::lambda-exp}
\lambda = \min \left(\frac{t}{32 K^4
\sum_{k=1}^m p_k  a_{kk}^2 },  \inv{4 e K^2\max_{k}\abs{a_{kk}}}\right)
\eens
Thus we have  
\bens\lefteqn{
\prob{\sum_{k=1}^m a_{kk} \xi_k X_k^2 - \sum_{k=1}^m a_{kk} p_k \E X_k^2  > t} } \\
&\le & 
\exp \left[-\inv{4e}\min\left(\frac{t^2}{3 K^4  \sum_{k=1}^m p_k a_{kk}^2},
\frac{t}{K^2 \max_{k} \abs{a_{kk}}} \right)\right].
\eens
We note that these constants have not been optimized.
Repeating the arguments for $-A$ instead of $A$, we obtain
for every $t>0$, and for 
$S'_0 := \sum_{k=1}^m (-a_{kk}) \xi_k X_k^2 - \sum_{k=1}^m (- a_{kk}) p_k
\E X_k^2$
\bens
\lefteqn{
\prob{
\sum_{k=1}^m a_{kk} \xi_k X_k^2 - \sum_{k=1}^m a_{kk} p_k \E X_k^2  < - t} = 
\prob{S'_0 > t }}\\
& \le &
\exp \left[- \inv{4e}\min\left(\frac{t^2}{3 K^4 \sum_{k=1}^m p_k a_{kk}^2},
\frac{t}{K^2 \max_{k} \abs{a_{kk}}} \right)\right].
\eens
The lemma thus holds.
\end{proofof}

\noindent{\bf Part II: Off-diagonal Sum.}
We now focus on bounding the off-diagonal part of the sum:
\bens
\label{eq::defineS1}
S_{\offd} := \sum_{i\not=j}^m a_{ij} X_i X_j \xi_i \xi_j
\eens
 where by independence of $X$ and $\xi$,  
$\E S_{\offd} = \sum_{i\not=j}^m  a_{ij} \E X_i \E X_j \E \xi_i \E \xi_j = 0.$

We will show that the following large deviation inequality holds  for all $t > 0$,
\ben
\label{eq::offD}
\prob{\abs{S_{\offd}} > t} 
&  \le &  
2 \exp \left(- c\min\left(\frac{t^2}{K^4 \sum_{i\not=j} a_{ij}^2 p_i p_j},
    \frac{t}{K^2 \twonorm{A}} \right)\right) 
\een
First we prove a bound  on the moment  generating function for the
off-diagonal sum $S_{\offd}$.
We assume without loss of generality that $K = 1$ by replacing $X$
with $X/K$. Let $C_4$ be a constant to be specified. 
It holds that for all $\abs{\lambda} \le\inv{2\sqrt{C_4} \twonorm{A}}$
\ben
\label{eq::mgfboundS}
\E \exp(\lambda S_{\offd}) 
& \le &
\exp\left(1.44 C_4
 \lambda^2 \sum_{i\not=j}
  a_{ij}^2 p_i p_j\right).
\een
Thus we have  for $0 \le \lambda \le\inv{2\sqrt{C_4} \twonorm{A}}$ and  $t>0$, 
\bens
\label{eq::mgfSlarge}
\prob{S_{\offd} > t} \le 
\frac{\E \exp(\lambda S_{\offd})}{e^{\lambda t}} 
\le  
\exp\left(-\lambda t + 
1.44 C_4\lambda^2 \sum_{i\not=j} p_i p_j  a_{ij}^2 \right). 
\eens
Optimizing over $\lambda$, we conclude that 
\ben
\label{eq::lambda-quad-ber}
\prob{S_{\offd} > t} \le\exp \left(- c\min\left(\frac{t^2}{\sum_{i\not=j} a_{ij}^2 p_i p_j},
    \frac{t}{\twonorm{A}} \right)\right) =: q_1
\een
Repeating the arguments for $-A$ instead of $A$, we obtain for 
$S' := \sum_{i\not=j}^m (-a_{ij}) X_i X_j \xi_i \xi_j = -S_{\offd}$, 
 $0 \le \lambda \le\inv{2\sqrt{C_4} \twonorm{A}}$ and  $t>0$, 
\bens
 \prob{S' > t} \le 
\frac{\E \exp(\lambda S')}{e^{\lambda t}} 
=\frac{\E \exp(-\lambda S_{\offd})}{e^{\lambda t}} \le  
\exp\left(-\lambda t + 1.44 C_4\lambda^2 
\sum_{i\not=j} p_i p_j  a_{ij}^2 \right) \le q_1
\eens
by \eqref{eq::mgfboundS} and \eqref{eq::lambda-quad-ber}.
Thus we have 
\bens
\prob{\abs{S_{\offd}} > t} =\prob{S_{\offd} > t} + \prob{S_{\offd}<-t} = \prob{S_{\offd} > t} +
\prob{S'> t} = 2q_1.
\eens
Thus~\eqref{eq::offD} holds for all $t > 0$.
The theorem is thus proved by summing up the bad events for diagonal
sum and the non-diagonal sum 
while adjusting the constant $c$ in~\eqref{eq::q1}.

The proof of \eqref{eq::mgfboundS} follows essentially from the
decoupling and reduction arguments in~\cite{RV13} and thus omitted
from the main body of the paper.
For completeness, we include the full proof in Appendix~\ref{sec::HWapp}. 
See for example~\cite{PM95,PG99} for comprehensive discussions on modern decoupling methods.
\end{proofof2}

%% file: sparseproof.tex
\section{Proof of Theorem~\ref{thm::MNsparse}}
\label{sec::proofofMN}
\begin{proofof2}
Let $X, \xi$, $D_0$ and $D_{\xi}$ be defined as in
Theorem~\ref{thm::MNsparse}. 
We assume without loss of generality that
$K = 1$ by replacing $X$ with $X/K$. 
Denote by  $\xi = (\xi_1, \ldots, \xi_m) \in \{0, 1\}^m$ a random vector
with independent Bernoulli random variables $\xi_i$ such that
$\xi_i = 1$ with probability $p_i$ and $0$ otherwise.

We will bound the diagonal and the off-diagonal sums separately. 
Let $D_0 = [d_1, d_2, \ldots, d_m]$ be a symmetric matrix.
Recall that we need to estimate
\bens
q := \prob{\abs{X^T A_{\xi} X  - \expct{X ^T A_{\xi} X}}  > t}
\; \; \text{ where } \;\;A_{\xi}  =  D_0 D_{\xi}  D_0 =: (\tilde{a}_{ij})
\eens
We first separate the diagonal sum from the off-diagonal sum as follows:
\bens
\abs{X^T A_{\xi} X  - \expct{X ^T A_{\xi} X}} & \le & 
\abs{\sum_{i \not=j} X_i X_jA_{\xi, ij}}  + 
\abs{ \sum_{k=1}^m  X_k^2 A_{\xi, kk} - \E (X_k^2)\E( A_{\xi,  kk})} \\
& =: & \abs{S_{\off}} +\abs{S_{\diag}}
\eens
where $S_{\off}$ and $S_{\diag}$ denote the following random variables:
\bens
S_{\off} & := & \sum_{i \not= j} X_i X_{j} A_{\xi, i j}
= \sum_{i \not= j} X_i X_{j} \tilde{a}_{i j} \; \; \text{ and } \\
S_{\diag} & := & \sum_{k=1}^m  X_k^2 A_{\xi, kk} - \E (X_k^2) \E(A_{\xi,  kk}).
\eens
To prove Lemma~\ref{lemma::expmgfS2-devi}, we need the following
bounds on moment  generating functions for the diagonal sum in $S_{\diag}$ in
Lemma~\ref{lemma::diagmgfbound} and 
the off-diagonal sum $S_{\offd}$ in Lemma~\ref{lemma::expmgfS1}.
Let $A_0 = D_0^2 =(a_{ij}) \succeq0$.  
The constants in the expression for $N$ (and $M$) are not being optimized:
\ben
\label{eq::defineN}
N & = &  82 \sum_{i=1}^m  a_{ii}^2 p_i + 
108  \sum_{i\not=j} a_{ij}^2 p_i p_j, 
\een
\begin{lemma}
\label{lemma::diagmgfbound}
For all  $\abs{\lambda} \le \inv{128 \twonorm{A_0}}$,
\bens
\E \exp(\lambda S_{\diag}) \le \exp\big(\lambda^2 N \big)  \; \text{ and } \; 
\E \exp(-\lambda S_{\diag}) \le  \exp\big(\lambda^2 N \big).
\eens
\end{lemma}
To prove Lemma~\ref{lemma::diagmgfbound}, first we write $S_{\diag} = S_0 + S_{\star}$ where 
\ben
\label{eq::defineS0}
S_0 & := & \sum_{k=1}^m (X_k^2 -\E (X_k^2)) A_{\xi, kk} 
 = \sum_{k=1}^m (X_k^2 - \E (X_k^2)) 
(\sum_{\ell=1}^m d^2_{k \ell} \xi_{\ell} ),\\
\label{eq::defineS2}
S_{\star} & := & 
 \sum_{k=1}^m \E (X_k^2) A_{\xi, kk} - \E (X_k^2)\E(A_{\xi,  kk})
= \sum_{k=1}^m \E (X_k^2)
(\sum_{\ell=1}^m d^2_{k \ell} (\xi_{\ell} - \E\xi_{\ell} ))
\een
where recall 
\ben
\label{eq::defineD0}
A_{\xi} = D_0 D_{\xi} D_0, \; \text{ where} \; D_0 = [d_1,
\ldots, d_m].
\een
We now state the following bounds on  the moment generating functions
of $S_0$ and $S_{\star}$  in Lemmas~\ref{lemma::expmgfSd}
and~\ref{lemma::expmgfS0} respectively.  
The estimate on the moment generating function stated in Lemma~\ref{lemma::diagmgfbound} then follows 
immediately from the Cauchy-Schwartz inequality, 
in view of Lemmas~\ref{lemma::expmgfSd} and~\ref{lemma::expmgfS0}.
\begin{lemma}
\label{lemma::expmgfSd}
Let $a_{ii} = \twonorm{d_i}^2$ for $d_i$ as defined in \eqref{eq::defineD0}.
Let $a_{\infty} = \max_{i} \twonorm{d_i}^2$.
Then for $\abs{\lambda} < \inv{4 a_{\infty}}$,
\bens 
 \E \exp(\lambda S_{\star}) \le \exp\left(\half \lambda^2 e^{\abs{\lambda} a_{\infty}}\sum_{i=1}^m
  a_{ii}^2 \sigma^2_i \right) \; \;\text{ where } \; \; \E (X_k^2) \le \norm{X_k}_{\psi_2} = 1.
\eens
\end{lemma}

 \begin{proof}
We have by independence of $X$ and $\xi$ and by definition of
$S_{\star}$ in \eqref{eq::defineS2}
\bens
S_{\star}  
& = & \sum_{k=1}^m \E (X_k^2) \sum_{i=1}^m d_{ki}^2 (\xi_i - p_i) =  
\sum_{i=1}^m \left(\sum_{k=1}^m \E (X_k^2) d_{ki}^2\right) (\xi_i -p_i) =: \sum_{i=1}^m a'_{ii} (\xi_i - p_i).
\eens
where by assumption, we have $\E (X_k^2) \le \norm{X_k}_{\psi_2} \le K =1$ and hence
$$0 \le a'_{ii} := \sum_{k=1}^m
 \E (X_k^2) d_{ki}^2 \le a_{ii} \; \; 
\text{ and thus } \; \; \max_{i} \abs{a'_{ii}} \le a_{\infty}.$$
The bound on the mgf of $S_{\star}$ follows from
Lemma~\ref{lemma::bern-sum}. 
For $\abs{\lambda} < \inv{4 a_{\infty}}$, we have
\bens
\E \exp(\lambda S_{\star}) & = &  \E \exp(\lambda  \sum_{i=1}^m a'_{ii} (\xi_i - p_i))  \le  \exp\left(\half \lambda^2   e^{\abs{\lambda} a_{\infty}}\sum_{i=1}^m (a'_{ii})^2 \sigma^2_i \right) \\
&\le &  \exp\left(\half \lambda^2 e^{\abs{\lambda} a_{\infty}}\sum_{i=1}^m  a_{ii}^2 \sigma^2_i \right).
\eens
\end{proof}

\begin{lemma}
\label{lemma::expmgfS0}
Denote by $a_{ij} = \ip{d_i, d_j}$ for all $i \not=j$ and $a_{ii}=
\twonorm{d_i}^2$ for $d_i$ as defined in \eqref{eq::defineD0}. 
Denote by $a_{\infty} := \max_{i} a_{ii}$. Let $C_0 = 38.94$. 
Then for $\abs{\lambda} \le \inv{64 \twonorm{A_0}} \le  \inv{64  a_{\infty}}$,
\ben
\label{eq::final}
\E \exp(\lambda S_0) & \le &
 \exp\left(\lambda^2  (40 \sum_{j=1}^m  p_j a^2_{jj} + 54 \sum_{i\not=j} p_i p_j a_{ij}^2)\right).
\een
\end{lemma}

\begin{lemma}
\label{lemma::expmgfS1}
Let $A_0 = (a_{ij}) =D_0^2$. 
For all $\abs{\lambda} \le \inv{58 C \twonorm{A_0}}$ for 
some constant $C$
\bens
\E \exp(\lambda S_{\off})
& \le & 
\E \exp(\lambda^2 C_2 \xi^T (A_0 \circ A_0) \xi) \le 
\exp\big(\lambda^2 M \big) \\
\E \exp(-\lambda S_{\off}) & \le & \exp\big(\lambda^2 M \big)
\eens
where $C_2 = 32 C^2$ and $M = 11C^2 (3\sum_{i=1}^m p_i a^2_{ii} + 4
\sum_{i\not=j} a^2_{ij} p_i p_j)$.
\end{lemma}

We defer the proof of Lemma~\ref{lemma::expmgfS1} to
Section~\ref{sec::mgfSoffd} and Lemma~\ref{lemma::expmgfS0} to Section~\ref{sec::mgfS0}.
We are now ready to state the large deviation inequalities for the diagonal sum
$S_{\diag}$, followed by that for the off-diagonal sum $S_{\off}$.
\begin{lemma}
\label{lemma::expmgfS2-devi}
Let $A_0 = (a_{ij}) = D_0^2$.  
For all $t > 0$ and $N$ as defined in \eqref{eq::defineN},
\bens
\prob{\abs{S_{\diag}} > t/2} & \le &
2\exp \left(-\inv{16}\min\left(\frac{t^2}{N}, \frac{t}{32 \twonorm{A_0}} \right)\right).
\eens
\end{lemma}
For the off-diagonal sum, we now state the following large deviation bound as 
in Lemma~\ref{lemma::expmgfS1-devi}. 
\begin{lemma}
\label{lemma::expmgfS1-devi}
Suppose all conditions in Lemma~\ref{lemma::expmgfS1} hold.
For all $t >0$, and some large enough absolute constant $C$,
\bens
\prob{\abs{S_{\off}} > t/2} \le 
2\exp \left(-\inv{16}\min\left(\frac{t^2}{M}, 
\frac{t}{15 C \twonorm{A_0}} \right)\right)
\eens
where 
$M =  11C^2 (3\sum_{i=1}^m p_i a^2_{ii} + 4 \sum_{i\not=j} a^2_{ij} p_i p_j)$.
\end{lemma}

The Theorem is thus proved by summing up the two bad events:
\bens
q = \prob{\abs{S_{\diag} + S_{\off}} >t}
& \le & 
\prob{\abs{S_{\off}} > t/2} +
\prob{\abs{S_{\diag}} > t/2}
\eens
while adjusting the constant $c$
in~\eqref{eq::q1}.

It remains to prove Lemmas~~\ref{lemma::diagmgfbound}, \ref{lemma::expmgfS2-devi} and~\ref{lemma::expmgfS1-devi}.

\begin{proofof}{Lemma~\ref{lemma::diagmgfbound}}
Suppose that $\abs{\lambda} \le \inv{128\twonorm{A_0}}$.
By Lemmas~\ref{lemma::expmgfS0} and~\ref{lemma::expmgfSd}, 
\bens
\E^{1/2} \exp(2 \lambda S_{\star}) & \le &
\exp\left(\lambda^2 e^{2 \abs{\lambda} a_{\infty}}\sum_{j=1}^m \sigma^2_j
  a_{jj}^2\right) \\
\E^{1/2} \exp(2\lambda S_0) & \le &
 \exp\big(80\lambda^2 \sum_{j=1}^m  p_j a^2_{jj}\big)
\exp \left(108 \lambda^2 \sum_{i\not=j} a_{ij}^2 p_i p_j\right).
\eens
Now we have by the Cauchy-Schwartz inequality,
\bens
\lefteqn{
\E \exp\left(\lambda S_{\diag} \right) =
\E \exp\left(\lambda (S_{0} + S_{\star})\right)
 \le \E^{1/2}\exp(2 \lambda  S_0) \E^{1/2}\exp(2\lambda  S_{\star}) }\\
 &\le &  \exp\big(82 \lambda^2 \sum_{j=1}^m  \sigma^2_j a^2_{jj}\big)
\exp \left(108  \lambda^2 \sum_{i\not=j} a_{ij}^2 p_i p_j\right).
\eens 
\end{proofof}

\begin{proofof}{Lemma~\ref{lemma::expmgfS2-devi}}
Lemma~\ref{lemma::expmgfS2-devi} follows from
Lemma~\ref{lemma::diagmgfbound} immediately.
Let $\E_{X}$ and $\E_{\xi}$ denote the expectation with respect to random variables in vectors $X$  and $\xi$ respectively.

First, by the Markov's inequality, we have for 
$0 < \lambda \le \inv{128 \twonorm{A_0} }$
\bens
\prob{S_{\diag} > t/2} & = & \prob{\lambda S_{\diag} > \lambda t/2}
=  \prob{\exp(\lambda S_{\diag}) > \exp(\lambda t/2)} \\
& \le & 
 \frac{\E \exp(\lambda S_{\diag})}{e^{\lambda t/2}}
\le\exp(-\lambda t/2 + N \lambda^2 )
\eens
Optimizing over $\lambda$,  for which the optimal choice of $\lambda$ is $\lambda = \frac{t}{4 N}$. 
Thus, we have  for $t > 0$, 
\bens
\prob{S_{\diag} > t/2} & \le &
\exp \left(-\min\left(\frac{t^2}{16 N}, \frac{t}{4 *128 \twonorm{A_0}}
  \right)\right) \\
& \le &
\exp \left(-\inv{16}\min\left(\frac{t^2}{N}, \frac{t}{32 \twonorm{A_0}} \right)\right) 
=: q_d
\eens
Repeating the argument for $-A_{\xi}$ instead of $A_{\xi}$, we now consider
\bens
S'_{\diag} & := & \sum_{k=1}^m (X_k^2  (-A_{\xi, kk}) + \E (X_k^2) \E  (A_{\xi, kk}) 
= -S_{\diag}.
\eens
By Lemma~\ref{lemma::diagmgfbound}, we have for all $\abs{\lambda} \le \inv{128 \twonorm{A_0} }$
$$\E \exp(\lambda S'_{\diag}) =\E \exp(-\lambda S_{\diag}) \le  \exp\big(\lambda^2 N \big).$$
Thus, we have for $t > 0$ and $0 < \lambda \le \inv{128 \twonorm{A_0}}$,
\bens
\prob{S'_{\diag} > t/2}
& \le & 
 \frac{\E \exp(\lambda S'_{\diag})}{e^{\lambda t/2}}
\le\exp(-\lambda t/2 + N \lambda^2 ) \le q_d
\eens
The lemma is thus proved, given that for $t > 0$
\bens
\prob{S_{\diag} < -t/2} & = &\prob{S'_{\diag} > t/2} \le q_d \\
\prob{\abs{S_{\diag}} > t/2} & = & \prob{S_{\diag} > t/2} + 
\prob{S_{\diag} <- t/2}  \le 2q_d.
\eens
\end{proofof}

\begin{proofof}{Lemma~\ref{lemma::expmgfS1-devi}}
Lemma~\ref{lemma::expmgfS1-devi} follows immediately from Lemma~\ref{lemma::expmgfS1}.
We have for $0<\lambda \le \inv{58 C \twonorm{A_0}}$ and $S:= S_{\off}$,
\bens
\prob{S > t/2} = \prob{\exp(\lambda S) >\exp(\lambda t/2)} 
\le \frac{\E \exp(\lambda S)}{e^{\lambda t/2}}
\le \exp(-\lambda t/2 + M \lambda^2 )
\eens
for which the optimal choice of $\lambda$ is $\lambda = \frac{t}{4 M}$.
Thus we have for $t>0$,
\bens
\prob{S > t/2}  &\le & \exp(-\lambda t/2 + M \lambda^2 ) \\
&\le & 
 \exp \left(-\inv{16}\min\left(\frac{t^2}{M},\frac{t}{15 C  \twonorm{A_0}} \right)\right) =: q_{\offd}
\eens
Similarly, we have for $\lambda, t>0$,
\bens
\prob{S < -t/2} 
& = & \prob{-S > t/2} =
\prob{\exp(\lambda(- S)) >\exp(\lambda t/2)} \\ 
& \le &
\frac{\E \exp(\lambda(- S))}{e^{\lambda t/2}} \le
\exp(-\lambda t/2 + M \lambda^2 ) \le q_{\offd}.
\eens
The lemma is thus proved using the union bound.
\end{proofof}
The Theorem is thus proved.
\end{proofof2}

The plan is to first bound the moment  generating function for the
 $S_0$ in the diagonal sum in Section~\ref{sec::mgfS0}.
We then bound the moment  generating function for the
off-diagonal sum as stated in Lemma~\ref{lemma::expmgfS1} in Section~\ref{sec::mgfSoffd}.

\subsection{Proof of Lemma~\ref{lemma::expmgfS0}}
\label{sec::mgfS0}
\begin{proofof2}
Recall $A_{\xi} = D_0 D_{\xi} D_0 = (\tilde{a}_{ij}) = (d_i^T D_{\xi} d_j)$.
Then for $\tilde{a}_{kk} =d_k^T D_{\xi} d_k = \sum_{i=1}^m d_{ki}^2 \xi_i$
\bens
S_0 & := & 
\sum_{k=1}^m (X_k^2 -\E X_k^2) A_{\xi, kk} = \sum_{k=1}^m (X_k^2 -\E X_k^2) \tilde{a}_{kk}
\eens
To estimate the moment generating function of $S_0$, we first 
consider $\xi$ as being fixed and thus treat $\tilde{a}_{ij}$ as fixed
coefficients.  The bound on  the moment generating function of $S_0$ as in
\eqref{eq::defineS0} 
will involve  the following symmetric matrices $A_1$ and $A_2$ which
we now define:
\ben
\nonumber
\label{eq::eigendec1}
A_1 & := & D_0 \circ  D_0 = [d_1 \circ d_1, \ldots, d_m \circ d_m], \\
\label{eq::defineA2}
A_2  =(a''_{ij}) & = &  A_1^2 = \left(d_1 \circ d_1, \ldots, d_{m} \circ d_m\right)\left(d_1
    \circ d_1, \ldots, d_{m} \circ d_m\right)^T \\
\nonumber
 & = &  \sum_{k=1}^m (d_k d_{k}^T)\circ (d_k d_{k}^T )
=\sum_{k=1}^m (d_k \circ d_k) ( d_{k} \circ d_k)^T 
\succeq 0.
\een
Thus we have both $A_0, A_2$
being positive semidefinite, while in general $A_1$ is not positive
semidefinite unless $D_0 \succeq 0$ by the Schur Product Theorem. See
Theorem 5.2.1~\cite{HJ91}. 
\begin{lemma}
\label{lemma::expmgfS0local}
Suppose all conditions in Lemma~\ref{lemma::expmgfS0} hold.
Let $C_0 = 38.94$. 
Then for $\abs{\lambda} \le \inv{64 \twonorm{A_0}} \le  \inv{64  a_{\infty}}$,
\ben
\label{eq::tightS0}
\E \exp(\lambda S_0) & \le &   
\E \exp(C_0 \lambda^2 \xi^T  A_2 \xi) \le   \E \exp(C_0 \lambda^2 \fnorm{\diag(A_{\xi})}^2) .
\een
\end{lemma}

\begin{proof}
We first compute the moment generating function for $S_0$ when
$\xi$ is fixed.
Conditioned on $\xi$, $\tilde{a}_{kk}, \forall k$ are considered as
fixed coefficients. 
Indeed, for $\abs{\lambda} \le \inv{64 a_{\infty}}$, by independence of $X_i$
\bens
\E \left(\exp(\lambda S_0) | \xi\right)
& = & \E_X \exp\left(\lambda \sum_{k=1}^m  \tilde{a}_{kk}  (X_k^2 - \E X_k^2)\right) = 
\prod_{k=1}^m \E_X \exp(\lambda \tilde{a}_{kk} ( X_k^2 -\E X_k^2)) \\
& \le &  
\prod_{k=1}^m \exp\left(38.94 \lambda^2 \tilde{a}_{kk}^2
\right) =  \exp\big(C_0 \lambda^2\sum_{k=1}^m \tilde{a}_{kk}^2\big)
\eens
where the inequality follows from Lemma~\ref{lemma::expmgf2} with
$\tau := \lambda \tilde{a}_{kk}$ in view of \eqref{eq::lamb}:
\ben
\label{eq::lamb}
\forall k, \; \forall \xi,\;\;
\abs{\lambda  \tilde{a}_{kk}} \le \inv{64}
\le \inv{23.5 e} \text{ where} \; \; 
\abs{\tilde{a}_{kk}} \le \ip{d_k, d_k} = a_{kk} \le a_{\infty}
\een
Now
\bens
\sum_{k=1}^m \tilde{a}_{kk}^2
& = & \sum_{k=1}^m (d_{k}^T D_{\xi} d_k) ^2 = \sum_{k=1}^m \tr(d_{k}^T D_{\xi} d_k d_{k}^T D_{\xi} d_k)  \\
& = & \sum_{k=1}^m \tr(D_{\xi} d_k d_{k}^T D_{\xi} d_k d_{k}^T )  
=  \sum_{k=1}^m \xi^T( (d_k d_{k}^T)\circ d_k d_{k}^T ) \xi 
 =: \xi^T A_2 \xi
\eens
where $A_2 =(a_{ij}'') = (D_0 \circ D_0)^2$ is as defined in~\eqref{eq::defineA2}.
Thus
\ben
\label{eq::antistar}
\E_X \exp(\lambda S_0) & \le &   \exp(C_0 \lambda^2 \xi^T  A_2 \xi) =  
\exp\big(C_0  \lambda^2  \fnorm{\diag(A_{\xi})}^2\big)
\een
and \eqref{eq::tightS0}  is thus proved by taking
expectation on both sides of \eqref{eq::antistar} with respect to 
random variables in vector $\xi$.
\end{proof}
To prove~\eqref{eq::final}  in the Lemma statement, 
notice that for all $\xi \in \{0, 1\}^m$,
\bens
\sum_{k=1}^m \tilde{a}_{kk}^2 = \fnorm{\diag(A_{\xi})}^2 \le 
\fnorm{A_{\xi}}^2.
\eens
Thus we have for  $\abs{\lambda} \le \inv{64 \twonorm{A_0}}$,
\bens
\E \exp(\lambda S_0)
&  \le & 
\E \exp(C_0\lambda^2 \fnorm{\diag(A_{\xi})}^2) \\
&  \le & 
\E \exp(C_0\lambda^2 \fnorm{A_{\xi}}^2) =\E \exp(C_0 \lambda^2 \xi^T (A_0 \circ A_0) \xi)
\eens 
where $A_0 = (a_{ij})$. 
Finally, we invoke Corollary~\ref{coro::actsharp} to finish the proof
of Lemma~\ref{lemma::expmgfS0}.
\begin{corollary}
\label{coro::actsharp}
Let $A_0, \xi$ be as defined in Theorem~\ref{thm::MNsparse}.
Then for $\abs{\lambda} \le  \inv{64 \twonorm{A_0}}$ and $C_0 \le  38.94$
\bens
\E \exp\left(C_0 \lambda^2 \sum_{i,j} a^2_{ij} \xi_i \xi_j\right) 
& \le &  \exp\big(\lambda^2 N \big)
\eens
where $N=  \big(40 \sum_{j=1}^m  p_j a^2_{jj} + 54 \sum_{i\not=j} p_i p_j a_{ij}^2\big)\big)$.
\end{corollary}
The proof of Corollary~\ref{coro::actsharp} follows exactly that of 
Corollary~\ref{coro::action} in view of Theorem~\ref{thm::Bernmgf} and
is thus omitted. The Lemma is thus proved.
\end{proofof2}

\begin{remark}
An alternative bound can be stated as follows: for $\lambda \le \inv{64 a_{\infty}}$,
\bens
\E \exp(\lambda S_0) & \le & 
\exp \left(41 \lambda^2 \sum_{j=1}^m  \sigma^2_j a^2_{jj} + 52\lambda^2 \twonorm{A_1 \vecp}^2 \right)
\eens
where $\vecp =[p_1, \ldots, p_m]$ and $\sigma_j^2 = p_j(1-p_j)$.
The proof follows from a direct analysis based on  the quadratic form
$\xi^T A_2 \xi$ on the RHS of \eqref{eq::tightS0}, which is omitted from the present
paper. This bound may lead to a slight improvement upon the final
bound in \eqref{eq::final}. We do not pursue this improvement here because the bound in~\eqref{eq::final} is 
sufficient for us to obtain the final large deviation bound as stated in Theorem~\ref{thm::MNsparse}.
\end{remark}

\subsection{Proof of Lemma~\ref{lemma::expmgfS1}}
\label{sec::mgfSoffd}
\begin{proofof2}
Let $\E_{X}$ and $\E_{\xi}$ denote the expectation with respect to random variables in vectors $X$  and $\xi$ respectively.
Recall
\bens
S_{\off} = \sum_{i \not= j} X_i X_{j} (A_{\xi, i j}) =: \sum_{i \not=
  j} \tilde{a}_{ij} X_i X_j 
\; \; \text{ where } \; \; 
 \tilde{a}_{ij} =  d_i^T D_{\xi} d_j= \sum_{k=1}^m d_{ik} \xi_k d_{j k}
\eens
To estimate the moment generating function of $S_{\off}$, we first 
consider $\xi$ as being fixed and thus treat $\tilde{a}_{ij}$ as fixed
coefficients.  Lemma~\ref{lemma::RRbound} reduces the original problem of 
estimating the moment generating function of $S_{\off}$ to the new
problem of estimating the moment generating function of $S := \xi^T
(A_0 \circ A_0)\xi$, which involves a new quadratic form with 
independent non-centered random variables $\xi_1, \ldots, \xi_m \in
\{0, 1\}$ and the symmetric matrix $(A_0 \circ A_0)$ as shown in \eqref{eq::RRbound}.
Lemma~\ref{lemma::RRbound} follows from the proof of Theorem
1~\cite{RV13} directly. We omit the proof in this paper.
\begin{lemma}
\label{lemma::RRbound}
Consider $\xi \in \{0, 1\}^m$ as being fixed and denote by 
$A_{\xi} = D_0 D_{\xi} D_0$ and $A_0 = D_0^2 = (a_{ij})$.
Then, for some constant $C$ and $\abs{\lambda} \le \inv{12C
  \twonorm{A_0}}$ and $ C_2 = 32 C^2$,
\ben
\label{eq::RRbound}
\E_X \exp(\lambda S_{\off})  \le
 \exp\left(C_2 \lambda^2
  \fnorm{A_{\xi}}^2\right) =  \exp\left(C_2 \lambda^2 \xi^T (A_0 \circ A_0) \xi \right). 
\een
\end{lemma}
Note that $\twonorm{D_{\xi} D_0} \le\twonorm{D_0}$ and hence by symmetry
\bens
 \twonorm{A_{\xi}} &  = &
\twonorm{D_0 D_{\xi} D_{\xi} D_0} = \twonorm{D_{\xi} D_0}^2  
\le \twonorm{D_{\xi}}^2\twonorm{D_0}^2 = \twonorm{A_0} \\
\fnorm{A_{\xi}}^2 & = & \tr(A_0 D_{\xi} A_0 D_{\xi}) = \xi^T (A_0 \circ A_0) \xi.
\eens 
In order to estimate the moment generating function
$S_{\offd}$,  we now take expectation with respect to $\xi$ on both sides
of~\eqref{eq::RRbound}. 
Thus we have for  $\abs{\lambda} \le \inv{12C \twonorm{A_0}}$
\ben
\label{eq::localbound}
\E_{\xi}\E_X\exp(\lambda S_{\off}) 
& \le & \E \exp\left(C_2 \lambda^2 \fnorm{A_{\xi}}^2\right) =  \E\exp\left(C_2 \lambda^2 \xi^T (A_0 \circ A_0)\xi \right).
\een
\begin{corollary}
\label{coro::action}
Then for $\abs{\lambda} \le  \inv{58C \twonorm{A_0}}$ and $t := C_2
\lambda^2$, where $C_2 = 32 C^2$ and  $C$ is a large enough absolute constant, 
\bens
\E \exp\left(t \sum_{i,j} a^2_{ij} \xi_i \xi_j\right) 
& \le &  \exp\big(\lambda^2 M\big)
\eens
where $M := C^2 \big(33\sum_{i=1}^m a^2_{ii} p_i +  44 \sum_{i\not=j}
a_{ij}^2 p_i p_j \big)$.
\end{corollary}
Combining Lemma~\ref{lemma::RRbound},~\eqref{eq::localbound} and Corollary~\ref{coro::action},
we have for $\abs{\lambda} \le \inv{58 C \twonorm{A_0}}$
\bens
\E(\lambda S_{\off}) 
\le \E \exp\left(C_2 \lambda^2  \fnorm{A_{\xi}}^2\right) =
\E \exp\left(t\sum_{i,j} a^2_{ij} \xi_i \xi_j\right)  \le
\exp(\lambda^2 M).
\eens
Lemma~\ref{lemma::expmgfS1} thus holds. 
\end{proofof2}

\silent{
Hence by Theorem \ref{thm::Bernmgf},
we have for
\bens
\nonumber
\lefteqn{
\E \exp\big(t \sum_{i,j} a^2_{ij}  \xi_i \xi_j \big)
\le  \exp t\big(\sum_{j=1}^m  p_j a^2_{jj} + \sum_{i \not=j} a^2_{ij} p_i p_j\big)
\E \exp t S} \\
&\le &
\exp \big(1.02 t  \sum_{j=1}^m  p_j a^2_{jj} + 
1.38 t \sum_{j \not=i} a^2_{ij} p_i p_j \big) \\
& \le & \exp\big(C^2 \lambda^2 \big(33\sum_{j=1}^m  p_j a^2_{jj} +
44 \sum_{i\not=j} p_i p_j a_{ij}^2\big)\big) \le \exp\big(\lambda^2 M\big)
\eens
where  for $S:= \sum_{j=1}^m  a^2_{jj} (\xi_{j}- p_{j}) + 
\sum_{j=1}^m \sum_{i\not=j} a^2_{ij} (\xi_i  \xi_{j}- p_i p_{j})$
\bens
\sum_{i,j} a^2_{ij} \xi_i \xi_j = S + \E \sum_{i,j} a^2_{ij} \xi_i \xi_j
\; \; \text{ where } \; \; \E \sum_{i,j} a^2_{ij} \xi_i \xi_j 
= \sum_{j=1}^m  a^2_{ij} p_j +\sum_{i\not=j}  a^2_{ij} p_i p_j
\eens
by independence of $\xi_i$ and $\xi_j$.}

Corollary~\ref{coro::action} follows from Theorem~\ref{thm::Bernmgf} immediately,
which is derived in the current paper for estimating the moment generating function of $S' :=
\xi^T A \xi$ where $A$ is an arbitrary matrix and $\xi$ is a Bernoulli random vector with independent elements as
defined in Theorem~\ref{thm::MNsparse}.

\begin{proofof}{Corollary~\ref{coro::action}}
Clearly for the choices of $t$ and $\lambda$,
\bens
t = C_2\lambda^2 & \le & \frac{32 C^2}{58^2 C^2 \twonorm{A_0}^2} 
\le  \inv{104 \twonorm{A_0}^2} \le 
\inv{104 \norm{A_0 \circ A_0}_1  \bigvee \norm{A_0 \circ A_0}_{\infty}},
\eens
where we use the fact that for symmetric $A_0$,
\bens
\norm{A_0 \circ A_0}_1 = \norm{A_0 \circ A_0}_{\infty}
= \max_{1 \le i\le m} \sum_{j=1}^m a_{ij}^2  = \max_{1 \le i \le m} \twonorm{A_0 e_i}^2 
\le \twonorm{A_0}^2.
\eens
Thus we can apply Theorem~\ref{thm::Bernmgf} with $B := (A_0 \circ
A_0)$ to obtain for $0< t \le \inv{104(\norm{A}_1 \vee \norm{A}_{\infty})}$,
\bens
\nonumber
\E \exp\left(t\sum_{i,j} a^2_{ij} \xi_i \xi_j\right) 
& \le & \exp\big(1.02 t \sum_{j=1}^m  p_j a^2_{jj}\big) 
\exp\left(1.373 t\sum_{i\not= j} a^2_{ij}  p_i p_j \right) \\
\label{eq::tbound}
& \le &
 \exp\big(C^2 \lambda^2
 \big(33\sum_{j=1}^m  p_j a^2_{jj} +
44 \sum_{i\not=j} p_i p_j a_{ij}^2\big)\big).
\eens
\end{proofof}

\silent{
\begin{proofof}{Corollary~\ref{coro::actsharp}}
Clearly for the choice of $\lambda$,
\bens
t := C_0 \lambda^2 & \le & \frac{38.94}{64^2  \twonorm{A_0}^2} 
\le  \inv{104 \twonorm{A_0}^2} \le 
\inv{104 \norm{A_0 \circ A_0}_1  \bigvee \norm{A_0 \circ A_0}_{\infty}},
\eens
where we use the fact that for symmetric $A_0$,
\bens
\norm{A_0 \circ A_0}_1 = \norm{A_0 \circ A_0}_{\infty}
= \max_{1 \le i\le m} \sum_{j=1}^m a_{ij}^2  = \max_{1 \le i \le m} \twonorm{A_0 e_i}^2 
\le \twonorm{A_0}^2.
\eens
Thus we can apply Theorem~\ref{thm::Bernmgf} with $B := (A_0 \circ A_0)$
 to obtain for $0< t \le \inv{104(\norm{B}_1 \vee \norm{B}_{\infty})}$,
\ben
\nonumber
\E \exp\left(t\sum_{i,j} a^2_{ij} \xi_i \xi_j\right) 
& \le & \exp\big(1.02 t \sum_{j=1}^m  p_j a^2_{jj}\big) 
\exp\left(1.373 t\sum_{i\not= j} a^2_{ij}  p_i p_j \right) \\
\label{eq::tbound}
& \le & \exp\big( \lambda^2  \big(40 \sum_{j=1}^m  p_j a^2_{jj} + 54 \sum_{i\not=j} p_i p_j
 a_{ij}^2\big)\big) 
\een
\end{proofof}}

%% file: noncenter.tex
\section{Proof of Theorem~\ref{thm::Bernmgf}}
\label{sec::noncenter}
We first state the following Theorem~\ref{thm::decouple} from a note by 
Vershynin~\cite{VerDecouple}; we state its consequence as follows.
\begin{theorem}
\label{thm::decouple}
Let $A$ be an $m \times m$ matrix.
Let $X =(X_1, \ldots, X_m)$ be a random vector with independent
mean zero coefficients. Then, for every convex function $F$,
\ben
\label{eq::decoupled}
\E F(\sum_{i\not=j} a_{ij}X_i X_j)
 \le \E F(4 \sum_{i\not=j} a_{ij}X_i X'_j).
\een
where $X'$ is an  independent copy of $X$.
\end{theorem}
Let $Z_i := \xi_i - p_i$. Denote by $\sigma_i^2 = p_i(1-p_i)$.
For all  $Z_i$, we have $\abs{Z_i} \le 1$, $\E Z_i = 0$ and 
\ben
\label{eq::Z2}
\E Z_i^2 & = & (1-p_i)^2 p_i + p_i^2 (1-p_i) = p_i(1-p_i) =\sigma^2_i,
\\
\label{eq::Zabs}
\E \abs{Z_i} & = & (1-p_i) p_i + p_i(1-p_i) = 2p_i(1-p_i) = 2 \sigma^2_i.
\een

\begin{proofof}{Theorem~\ref{thm::Bernmgf}}
Let  $Z_i = \xi_i -p_i$. 
Denote by $\breve{a}_{i} := \sum_{j\not=i} (a_{ij} + a_{ji}) p_j + a_{ii}$.
We express the quadratic form as follows:  
\bens
\sum_{i=1}^m a_{ii} (\xi_i - p_i) 
+ \sum_{i\not=j} a_{i j} (\xi_i  \xi_{j}- p_i p_{j})
& = &
\sum_{i\not=j} a_{ij} Z_i Z_j + \sum_{j=1}^m Z_j \breve{a}_{i} =: S_1
+ S_2.
\eens
We first state the following bounds on  the moment generating functions
of $S_1$ and $S_2$  in \eqref{eq::mgS1} and \eqref{eq::mgS2}.
The estimate on the moment generating function for $\sum_{i,j} a_{ij}
\xi_i \xi_j$ then follows 
immediately from the Cauchy-Schwartz inequality in view of \eqref{eq::mgS1} and \eqref{eq::mgS2}. 

\noindent{\bf Bounding  the moment generating function for $S_1$.}
In order to bound the moment generating function for $S_1$, 
we start by a decoupling step following
Theorem~\ref{thm::decouple}. Let  $Z'$ be an  independent copy of $Z$. \\

\noindent{\bf Decoupling.} 
Now consider random variable $S_1  :=  \sum_{i\not=j } a_{ij} (\xi_i
-p_i)(\xi_j -p_j)  = \sum_{i\not=j } a_{ij} Z_i Z_j$ and
\bens 
S'_1 := \sum_{i \not=j} a_{ij} Z_i Z'_j, 
\; \text{ we have } \; \;
\E \exp(2\lambda S_1)  \le \E \exp(8\lambda S'_1) =: f
\eens
by~\eqref{eq::decoupled}. Thus we have by independence of $Z_i$,
\ben
\label{eq::precur}
f & := &  \E_{Z'}\E_{Z}\exp\left(8 \lambda  \sum_{i=1}^m Z_i  \sum_{j
    \not=i} a_{ij} Z'_j\right) =
\E_{Z'} \prod_{i=1}^m \E \left(\exp\left(8 \lambda Z_i 
\tilde{a}_{i}\right)\right).
\een
First consider $Z'$ being fixed. Let us define 
\bens
t_i := 8 \lambda \tilde{a}_i \;\; \text{ where } \; \; 
\tilde{a}_{i} := \sum_{j\not=i} a_{ij} Z'_j.
\eens
Hence for all $0 \le \lambda \le \inv{104 \norm{A}_{\infty} }$
 and $C_4 := \frac{4}{13} e^{1/13}$, and any given fixed $Z'$ by~\eqref{eq::elem}
\ben
\nonumber
\E\exp\left(8 \lambda  \tilde{a}_{i} Z_i\right)  
& := & \E\exp\left(t_i Z_i\right) \le 
1 + \inv{2} t_i^2 \E Z_i^2 e^{\abs{t_i}} \le
\exp\left(\inv{2} t_i^2 \E Z_i^2 e^{\abs{t_i}} \right) \\
& \le & 
\label{eq::oneterm}
\exp\left(\frac{4}{13} e^{1/13} \lambda \abs{\tilde{a}_{i}} 
 \sigma^2_i \right) =:
\exp\left(C_4 \lambda \abs{\tilde{a}_{i}} \sigma^2_i \right)
\een
where $Z_i, \forall i$ satisfies: $\abs{Z_i} \le 1$, $\E Z_i = 0$ and 
$\E Z_i^2 = \sigma_i^2$, 
\bens
\abs{t_i} = \abs{8 \lambda \tilde{a}_{i}} 
\le 8 
\lambda\sum_{j\not=i} \abs{a_{ij}}  \abs{Z'_j}  \le 8  \lambda
\norm{A}_{\infty}  \le \inv{13} \; 
\; \text{ and } \quad \; \inv{2} t_i^2 \le \frac{4}{13} \lambda
\abs{\tilde{a}_i}.
\eens
 Denote by $\abs{\bar{a}_{j}} := \sum_{i \not=j}
\abs{a_{ij}} \sigma^2_i$. Thus by \eqref{eq::precur} and \eqref{eq::oneterm}
\bens
f & \le &  
\E_{Z'} \prod_{i=1}^m
\exp\left(C_4 \lambda \abs{\tilde{a}_{i}}  \sigma^2_i\right) \le 
\E_{Z'} \exp\left(C_4\lambda\sum_{i=1}^m  
\sigma^2_i \sum_{j\not=i} \abs{a_{ij}}\abs{Z'_j}\right) \\
& = & 
\prod_{j=1}^m \E \exp\left(C_4\lambda\abs{Z'_j} 
\sum_{i \not=j}^m \abs{a_{ij}} \sigma^2_i \right) 
=: \prod_{j =1}^m \E\exp\left(C_4 \lambda  \abs{\bar{a}_{j}}\abs{Z'_j}
\right)
\eens
where we have by the elementary approximation~\eqref{eq::elem}
and $\breve{t}_j  := C_4 \lambda \abs{\bar{a}_{j}}$
\bens
\E\exp\left(C_4 \lambda  \abs{\bar{a}_{j}}\abs{Z'_j} \right) 
& =: & \E\exp\left(\breve{t}_j \abs{Z'_j} \right)
\le  1 + \E \left(\breve{t}_j \abs{Z'_j} \right) + 
\inv{2} (\breve{t}_j)^2 \E (Z'_j)^2 e^{\abs{\breve{t}_j}} \\
& \le &
 \exp\left(2 \breve{t}_j \sigma^2_j + \inv{2} (\breve{t}_j)^2  
\sigma^2_j e^{0.0008}\right)
\le 
 \exp\left(2.0005 \breve{t}_j \sigma^2_j\right)\\
& \le &
 \exp\left(2.0005  C_4 \lambda \abs{\bar{a}_{j}}   \sigma^2_j \right) 
 \le \exp\left(\frac{2}{3} \lambda  \sigma^2_j \sum_{i \not=j} \abs{a_{ij}} \sigma^2_i\right)
\eens
where $\E (Z'_i)^2 =\sigma^2_i$ and $\E \abs{Z'_i} = 2 \sigma^2_i$
following~\eqref{eq::Z2} and~\eqref{eq::Zabs},  and for  $0< \lambda \le \inv{104 \max(\norm{A}_1,\norm{A}_{\infty})}$,
\bens
\breve{t}_j := C_4 \lambda \abs{\bar{a}_{j}} 
= \frac{4}{13} e^{1/13} \lambda  \sum_{i \not=j} \abs{a_{ij}} \sigma^2_i
\le \frac{4}{13} e^{1/13} \inv{4}
\frac{\sum_{i} \abs{a_{ij}}}{104 \norm{A}_{1}} 
\le \inv{13} e^{1/13} \inv{104}< 0.0008.
\eens
Thus for every $0 \le \lambda \le \inv{104\max(\norm{A}_1, \norm{A}_{\infty})}$,
\ben
\label{eq::mgS1}
\E \exp(\lambda 2S_1)
& \le & \exp \left(\frac{2}{3}\lambda \sum_{i\not=j} 
\abs{a_{ij}} \sigma^2_i  \sigma^2_j \right).
\een
\noindent{\bf Bounding  the moment generating function for $S_2$.}
Recall
\bens
S_2 & := & 
\sum_{i=1}^m Z_i \left(\sum_{j\not=i} (a_{ij} + a_{ji}) p_{j} +
a_{ii}\right) =: \sum_{i=1}^m Z_i \breve{a}_{i}.
\eens
Let $a_{\infty} := \max_{i} \abs{\breve{a}_{i}} \le \norm{A}_{\infty} +
\norm{A}_1$. 
Thus we have by Lemma~\ref{lemma::bern-sum}
\bens
g:= \E \exp(2 \lambda S_2) & = & 
\E\exp\left(2\lambda\sum_{i=1}^m Z_i \breve{a}_{i}\right) \\
\label{eq::s2bound}
& \le & \exp\left(2 \lambda^2   e^{2\abs{\lambda} a_{\infty}} 
\sum_{i=1}^m  \breve{a}_{i}^2 \sigma^2_i \right) 
 \le  \exp\left(C_5 \abs{\lambda} \sum_{i=1}^m  \abs{\breve{a}_{i}} p_i \right)
\eens
where $e^{2\lambda a_{\infty}} 2 \lambda \abs{\breve{a}_i} \le \inv{26}
e^{1/26} =: C_5 \le 0.04$ given that for all $\abs{\lambda} \le \inv{52 (\norm{A}_{\infty} +
  \norm{A}_1)}$
\bens
2 \lambda \abs{\breve{a}_{i}} \le \frac{ 2(\norm{A}_{\infty} + \norm{A}_1)}
{52 (\norm{A}_{\infty} + \norm{A}_1)} \le \inv{26} \; \;  \text{ for all $i$}.
\eens
Thus we have for $0 < \lambda\le \inv{52 (\norm{A}_{\infty} +
  \norm{A}_1)}$,
\ben
\label{eq::mgS2}
\E \exp(\lambda 2 S_2) 
& \le &
\exp\left(0.02 * 2\lambda{\sum_{i=1}^m p_i \abs{\breve{a}_i} }\right).
\een
Hence by the Cauchy-Schwartz inequality,  for all 
$0< \lambda \le \inv{104 (\norm{A}_{\infty} \vee  \norm{A}_1)}$,
\bens
\lefteqn{
\E \exp\left(\lambda \left(\sum_{i=1}^m a_{ii} (\xi_i - p_i) 
+ \sum_{i\not= j} a_{ij}  (\xi_i \xi_j - p_i p_j )\right)\right)}\\
& =& \E \exp\left(\lambda ( S_1 + S_2)\right) \le
\E^{1/2}\exp(2 \lambda  S_1) \E^{1/2}\exp(2\lambda  S_2)
\eens
The theorem is thus proved by multiplying 
$ \exp\left(\lambda \left(\sum_{i=1}^m a_{ii} p_i + \sum_{i\not= j}
    a_{ij} p_i p_j\right)\right)$ on both sides of the above
inequality.
\end{proofof}

%% file: gaussian.tex
\section{Proof of \eqref{eq::mgfboundS} }
\label{sec::HWapp}
\begin{proofof2}
The proof structure follows from the proof of Theorem~\ref{thm::HW}~\cite{RV13}.
Recall $S := \sum_{i\not=j}^m a_{ij} X_i X_j \xi_i \xi_j$.
We start with a decoupling step. \\

\noindent{\bf Step 1. Decoupling.} 
Let $\delta = (\delta_1, \ldots, \delta_m) \in \{0, 1\}^m$ be a random vector
with independent Bernoulli random variables with  
$\E \delta_i = 1/2$, which is independent of $X$ and $\xi$.
Let $X_{\Lambda_{\delta}}$ denote $(X_i)_{i \in \Lambda_{\delta}}$ for
a set $\Lambda_{\delta} := \{ i \in [m] : \delta_i = 1\}$.
Let $\E_{X}$, $\E_{\xi}$ and 
$\E_{\delta}$ denote the expectation with respect to random variables in $X$, $\xi$ and $\delta$ respectively.
Now consider random variable
\bens 
S_{\delta} := \sum_{i,j} \delta_i (1-\delta_j) a_{ij} X_i 
X_j \xi_i \xi_j
\; \text{ and hence} \; 
S = 4 \E_{\delta} S_{\delta}.
\eens
By Jensen's inequality, for all $\lambda \in \R$,
\ben
\label{eq::Jensen}
\E \exp(\lambda S)
= \E_{\xi}  \E_{X} \exp(\E_{\delta} 4 \lambda S_{\delta})
 \le \E_{\xi}  \E_{X} 
\E_{\delta} \exp(4 \lambda S_{\delta}).
\een
where the last step holds because $e^{a x}$ is convex on $\R$
for any $a \in \R$.

Consider $\Lambda_{\delta} :=\{i\in[m]: \delta_i = 1\}$.
Denote by $f(\xi, \delta, X_{\Lambda_{\delta}})$ the conditional
moment generating function of random variable $4S_{\delta}$:
$$f(\xi, \delta, X_{\Lambda_{\delta}}) 
:=\E\left(\exp(4 \lambda S_\delta) |\xi, \delta,
  X_{\Lambda_{\delta}}\right).$$
Conditioned upon $X_{\Lambda_{\delta}}$ for a fixed realization of $\xi$ and $\delta$,
we rewrite $S_{\delta}$
\bens
S_{\delta} := \sum_{i\in \Lambda_{\delta}, j\in \Lambda_{\delta}^c} a_{ij} X_i X_j  \xi_i \xi_j=
\sum_{j\in \Lambda_{\delta}^c} X_j \left(\xi_j \sum_{i\in \Lambda_{\delta}} a_{ij} X_i \xi_i\right)
\eens
as a linear combination of mean-zero subgaussian random 
variables $X_j, j \in \Lambda_{\delta}^c$, with fixed coefficients.
Thus the conditional distribution of $S_{\delta}$ is sub-gaussian with
$\psi_2$ norm being upper bounded
by the $\ell_2$ norm of the coefficient vector $(\xi_j \sum_{i\in \Lambda_{\delta}} a_{ij} X_i \xi_i)_{j \in 
\Lambda_{\delta}^c}$~\cite{Vers12}(cf. Lemma 5.9).

Thus, conditioned upon $\xi, \delta$ and $X_{\Lambda_{\delta}}$, 
\ben
\label{eq::variance}
\norm{S_\delta}_{\psi_2} \le  C_0 \sigma_{\delta, \xi} \; \; 
\text{where }
\; \; \sigma_{\delta, \xi}^2 = 
\sum_{j\in \Lambda_{\delta}^c} \xi_j  \left(\sum_{i\in \Lambda_{\delta}} a_{ij} X_i \xi_i\right)^2.
\een
Thus we have for some large absolute $C >0$
\ben
\label{eq::mgf-subg-cond}
f(\xi, \delta,  X_{\Lambda_{\delta}})= \E\left(\exp(4 \lambda S_\delta) |\xi, \delta, X_{\Lambda_{\delta}}\right)
\le \exp(C \lambda^2 \norm{S_\delta}_{\psi_2}^2)
\le \exp(C' \lambda^2  \sigma_{\delta, \xi}^2).
\een
Taking the expectations of both sides with respect to 
$X_{\Lambda_{\delta}}$
 and $\xi$, we obtain
\ben
\label{eq::mgf-subg}
\E_{\xi} \E_{X_{\Lambda_{\delta}}}f(\xi, \delta, X_{\Lambda_{\delta}})
&  = & 
\E_{\xi} \E_{X_{\Lambda_{\delta}}} \E\left(\exp(4 \lambda S_\delta)
  |\xi, \delta, X_{\Lambda_{\delta}}\right) \\
\nonumber
& \le &  
\E_{\xi} \E_{X_{\Lambda_{\delta}}} \exp(C' \lambda^2  \sigma_{\delta,
  \xi}^2) =: \tilde{f}_{\delta}
\een
\noindent{\bf Step 2. Reduction to normal random variables.} 
Let $\delta, \xi$ and $X_{\Lambda_{\delta}}$ 
be a fixed realization of the random vectors defined as above.
Let $g=(g_1, \ldots, g_n)$, where $g_i \text{ i.i.d.} \sim N(0, 1)$.
Let $\E_{g}$ denote the expectation with respect to random
variables in $g$.
Consider random variable
\bens
Z := \sum_{j\in \Lambda_{\delta}^c} g_j \left(\xi_j \sum_{i\in
    \Lambda_{\delta}} a_{ij} X_i \xi_i\right) 
\eens
By the rotation invariance of normal distribution, for a fixed
realization of random vectors $\xi, \delta, X$, the conditional
distribution of $Z$ follows $N(0, \sigma_{\delta,
  \xi}^2)$ for $\sigma_{\delta, \xi}^2$ as defined
in~\eqref{eq::variance}.
Thus we obtain the conditional moment generating function for $Z$
denoted by
\bens
\E_g(\exp(t Z)) := \E(\exp(t Z)|\xi, \delta,X_{\Lambda_{\delta}}) = \exp(t^2 \sigma_{\delta, \xi}^2/2).
\eens
Choose $t = C_1 \lambda$ where $C_1 = \sqrt{2C'}$,  we have  
\bens
\E_g(\exp(C_1 \lambda Z)) = \exp(C' \lambda^2 \sigma_{\delta, \xi}^2)
\text{ which matches the RHS of~\eqref{eq::mgf-subg-cond}}.
\eens
Hence for a fixed realization of $\delta$, we can calculate $\f_{\delta}$ using $Z$ as follows:
\ben
\label{eq::redefine}
\tilde{f}_{\delta} := \E_{\xi} \E_X \exp(C' \lambda^2  \sigma_{\delta, \xi}^2)
= 
\E_{\xi}\E_{X}\E_g(\exp(C_1 \lambda Z)) =  \E \left(\exp(C_1 \lambda Z)|\delta \right).
\een
Conditioned on $\delta$, $\xi$ and $g$,
we can re-express $Z$:
\bens
Z & =  & \sum_{i\in \Lambda_{\delta}} X_i  \left(\xi_i \sum_{j\in
    \Lambda_{\delta}^c} a_{ij} g_j \xi_j \right) 
\eens
 as a linear combination of subgaussian random
variables $X_i, i \in \Lambda_{\delta}$ with fixed 
coefficients, which immediately imply that
\bens
\E(\exp(C_1 \lambda Z) | \delta,\xi, g)  
& \le &   \exp\left(C_3 \lambda^2  \sum_{i\in \Lambda_{\delta}}
\xi_i \left( \sum_{j\in \Lambda_{\delta}^c} a_{ij} g_j \xi_j \right)^2
\right).
\eens 
Let $P_{\delta}$ denote the coordinate projection of $\R^m$ onto $\R^{\Lambda_{\delta}}$.
Then conditioned on $\delta$, we have by definition of $\f_{\delta}$
as in~\eqref{eq::redefine} and the bounds on the conditional moment
generating function of $Z$ immediately above,
\ben
\nonumber
\tilde{f}_{\delta}
& = & \E \left(\exp(C_1 \lambda Z)|\delta \right) 
= 
\E_{\xi, g} \E \left(\exp(C_1 \lambda Z) | \delta, \xi, g\right)\\
& \le & 
\nonumber
\E \left[\exp\left(C_3 \lambda^2 \sum_{i\in \Lambda_{\delta}} \xi_i \left(\sum_{j\in
    \Lambda_{\delta}^c} a_{ij} g_j \xi_j  \right)^2 \right) | \delta\right] \\
& = & 
\nonumber
\E \left[ \exp\left(C_3 \lambda^2  \twonorm{D_{\xi} P_{\delta}  A (I-P_{\delta}) D_{\xi} g}^2\right) | \delta\right] \\
\label{eq::upperf}
& = & 
\E \left[\exp\left(C_3 \lambda^2  \twonorm{A_{\delta, \xi} g}^2\right) | \delta\right]
\een
where we denote by $A_{\delta, \xi} := D_{\xi} P_{\delta} A
(I-P_{\delta}) D_{\xi}$.
We will integrate $g$ out followed by $\xi$ in the next two steps.  \\

\noindent{\bf Step 3. Integrating out the normal random variables.}
Conditioned upon $\delta$ and $\xi$ and  by the rotation invariance of $g$, the random variables 
$\twonorm{A_{\delta, \xi} g}^2$ follows the same distribution as
$\sum_{i} s_i^2 g_i^2$ where $s_i$ denote the singular values of
$A_{\delta, \xi}$, with 
\ben
\nonumber
\max_i {s_i}
 & = &  \sqrt{\lambda_{\max}(A_{\delta,\xi}^T  A_{\delta,\xi})}
=: \twonorm{A_{\delta, \xi}} \le \twonorm{A}, \; \; \text{ and } \\
\sum_{i} s_i^2 = \fnorm{A_{\delta, \xi}}^2 
\nonumber
& = & \tr(A_{\delta, \xi}A_{\delta, \xi}^T)= 
\tr(D_{\xi} P_{\delta} A (I-P_{\delta}) D_{\xi} A^T P_{\delta} D_{\xi})\\
\label{eq::xiquad}
 & = & \tr(D_{\xi} P_{\delta} A (I-P_{\delta}) D_{\xi} A^T ) 
=\sum_{i\in \Lambda_{\delta}} \xi_i \sum_{j \in \Lambda_{\delta}^c} a_{ij}^2 \xi_j
\een
First we note that $g_i^2, \forall i$ follow the $\chi^2$ distribution 
with one degree of freedom, and $\E \exp(t g^2) = \inv{\sqrt{1-2t}}
\le e^{2t}$ for $t < 1/4$. Thus we have for a fixed realization of
$\delta, \xi$, and for all $\abs{\lambda} \le \inv{2\sqrt{C_3} \twonorm{A}}$, 
\bens
\E \left[ \exp(C_3 \lambda^2 s_i^2 g_i^2) | \delta, \xi\right] 
 =  \inv{\sqrt{1-2 C_3 \lambda^2  s_i^2}} \le \exp(2C_3 \lambda^2 s_i^2).
\eens
Hence for any fixed $\delta$ and $\xi$, 
and for $C_4 = 2C_3$ and $\abs{\lambda} \le \inv{2\sqrt{C_3} \twonorm{A}}$, 
we have by independence of $g_1, g_2, \ldots$, 
\ben
\nonumber
\E \left[\exp\left(C_3 \lambda^2  \twonorm{A_{\delta, \xi}
      g}^2\right) | \delta, \xi\right]  & = & 
\E \left[\exp\left(C_3 \lambda^2  \sum_{i} s_i^2 g_i^2\right) | \delta, \xi\right] \\
\nonumber
& = & \prod_{i} \E \left[ \exp(C_3 \lambda^2 s_i^2 g_i^2) | \delta,
  \xi\right] \\
\label{eq::normalbound}
& \le & \prod_{i} \exp(2C_3 \lambda^2 s_i^2).
\een
Thus we have by \eqref{eq::upperf}, \eqref{eq::xiquad} and \eqref{eq::normalbound}
\ben
\nonumber
\tilde{f}_{\delta} & \le & 
\E_{\xi}\E \left[\exp\left(C_3 \lambda^2  \twonorm{A_{\delta, \xi} g}^2\right) | \delta, \xi\right]
  \le \E \left[ \exp(2C_3 \lambda^2 \sum_{i} s_i^2) | \delta \right] \\
\label{eq::fdeltaupper}
& = &  
\E \left[\exp\left(C_4 \lambda^2 
 \sum_{i\in \Lambda_{\delta}} \xi_i  \sum_{j \in \Lambda_{\delta}^c}
 a_{ij}^2 \xi_j\right) | \delta \right].
\een
The key observation here is we are dealing with a quadratic form on
the RHS of \eqref{eq::fdeltaupper} which
is already decoupled thanks to the decoupling Step 1. \\

\noindent{\bf Step 4. Integrating out the Bernoulli random variables.}
For any given realization of $\delta$, we now need to bound the moment generating
function for the decoupled quadratic form on the RHS of \eqref{eq::fdeltaupper},
which is the content of Lemma~\ref{lemma::fdelta} where we take $t =
C_4 \lambda^2$ and conclude that for all $\delta$ and for all $\abs{\lambda} \le\inv{2\sqrt{C_4} \twonorm{A}}$,
$$\tilde{f}_{\delta} \le 
\exp\left(1.44 C_4 \lambda^2 \sum_{i\not=j}
  a_{ij}^2 p_i p_j\right).$$
\begin{lemma}
\label{lemma::fdelta}
Let $0< \tau \le \inv{4 \twonorm{A}^2}$. For any fixed realization of $\delta$,  we have
\bens
\E\left[\exp\left(\tau\sum_{i\in \Lambda_{\delta}} \xi_i 
\sum_{j \in \Lambda_{\delta}^c} a_{ij}^2 \xi_j\right) | \delta \right]
\le
\exp\left(1.44 \tau \sum_{i\not=j} a_{ij}^2 p_i p_j\right).
\eens
\end{lemma}

\begin{proof}
As mentioned, we are dealing with a quadratic form which
is already decoupled. Thus we integrate out $\xi_i$ for all $i \in
\Lambda_{\delta}$ followed by those in $\Lambda^c_{\delta}$.
Recall for any realization of $\delta$ and $0 < \tau \le \inv{4 \twonorm{A}^2}$
 we have  by independence of $\xi_1, \xi_2, \ldots$,
\ben
\nonumber
\lefteqn{
{f}_{\delta} :=
\E \left[\exp\left(\tau \sum_{i\in \Lambda_{\delta}} \xi_i 
\sum_{j \in \Lambda_{\delta}^c} a_{ij}^2 \xi_j\right) | \delta\right]} \\
\nonumber
& = & 
\E_{\xi_{\Lambda_{\delta}^c}} 
\E \left[\exp\left(\tau  \sum_{i\in \Lambda_{\delta}} \xi_i \sum_{j \in \Lambda_{\delta}^c} a_{ij}^2 \xi_j\right) |
\xi_{\Lambda_{\delta}^c}, \delta \right] \\
& = & 
\label{eq::start}
\E_{\xi_{\Lambda_{\delta}^c}} \prod_{i \in \Lambda_{\delta}}
 \E \left[\exp\left(\tau\xi_i \sum_{j \in \Lambda_{\delta}^c} a_{ij}^2 \xi_j\right) |
\xi_{\Lambda_{\delta}^c}, \delta \right]
\een


We will use the following approximation twice in our proof:
\ben
\label{eq::approx}
e^x -1 \le 1.2x \; \; \text{ which holds for } \; \; 0\le x \le 0.35.
\een
First notice that for all realizations of $\delta$ and $\xi$, we have for
$0< \tau \le \inv{4 \twonorm{A}^2}$
\bens
0\le \tau \sum_{j \in \Lambda_{\delta}^c}
 a_{ij}^2  \xi_j \le \tau \sum_{j} a_{ij}^2 \le 
\tau \twonorm{A}^2 \le 1/4
\eens
given that the maximum row $\ell_2$ norm of $A$ 
is bounded by the operator norm of matrix $A^T$:  $\twonorm{A^T} = \twonorm{A} = \sqrt{\lambda_{\max}(A^T  A)}$.
Hence, we have for $\abs{\lambda} \le \inv{2 \sqrt{C_4} \twonorm{A}}$,
\eqref{eq::approx} and the fact that $1+x \le e^x$,
\ben
\nonumber
\lefteqn{
\E \left[\exp\left(\tau \xi_i \sum_{j \in \Lambda_{\delta}^c} a_{ij}^2 \xi_j\right) |
\xi_{\Lambda_{\delta}^c}, \delta \right] = 
p_i\exp\left(\tau \sum_{j \in \Lambda_{\delta}^c} a_{ij}^2  \xi_j\right) + (1-p_i) }\\
& \le & 
\label{eq::quad-bernoulli}
p_i \left(1.2 \tau \sum_{j \in \Lambda_{\delta}^c} a_{ij}^2
  \xi_j\right) + 1
\le \exp\left(1.2 \tau p_i \sum_{j \in \Lambda_{\delta}^c} a_{ij}^2
  \xi_j\right).
\een
Thus we have by independence of $\xi_1, \xi_2, \ldots$, \eqref{eq::start} and \eqref{eq::quad-bernoulli} 
\ben
\nonumber
{f}_{\delta} & \le &
\E_{\xi_{\Lambda_{\delta}^c}} 
\prod_{i \in \Lambda_{\delta}} 
\exp\left(1.2 \tau p_i \sum_{j \in \Lambda_{\delta}^c} a_{ij}^2
  \xi_j\right) 
= \E_{\xi_{\Lambda_{\delta}^c}} 
\exp \left( \sum_{i \in \Lambda_{\delta}}  
1.2 \tau p_i \sum_{j \in \Lambda_{\delta}^c} a_{ij}^2 \xi_j
\right) \\
\nonumber
&  =  &
\E_{\xi_{\Lambda_{\delta}^c}} 
\exp \left(1.2 \tau \sum_{j \in \Lambda_{\delta}^c} \xi_j 
\sum_{i \in \Lambda_{\delta}}  a_{ij}^2 p_i\right)=
\prod_{j \in \Lambda_{\delta}^c} \E_{\xi_j} \exp \left( 1.2 \tau
\xi_j \sum_{i \in \Lambda_{\delta}}  a_{ij}^2 p_i\right)\\
&  =  &
\label{eq::end}
\prod_{j \in \Lambda_{\delta}^c} p_j \exp \left(1.2 \tau \sum_{i \in \Lambda_{\delta}}
  a_{ij}^2 p_i\right) + (1-p_j)
\een
where for all $\delta$ and $0< \tau \le \inv{4 \twonorm{A}^2}$, we have by the approximation in~\eqref{eq::approx}
\ben
\label{eq::approxII}
\exp\left(1.2 \tau \sum_{i \in \Lambda_{\delta}}
  a_{ij}^2 p_i\right) - 1  \le 1.44 \tau \sum_{i \in \Lambda_{\delta}}  a_{ij}^2 p_i
\een
given that  the column $\ell_2$ norm of $A$ is bounded by the operator
norm of $A$, and thus 
\bens 
1.2 \tau \sum_{i \in \Lambda_{\delta}}  a_{ij}^2 p_i
\le  1.2 \tau 
\sum_{i =1}^m a_{ij}^2 p_i 
\le 1.2 \sum_{i =1}^m a_{ij}^2 /(4\twonorm{A}^2) \le 0.3.
\eens
Now by \eqref{eq::end},~\eqref{eq::approxII} and the fact that $x + 1\le e^x$
\bens
\lefteqn{{f}_{\delta} \le 
\prod_{j \in \Lambda_{\delta}^c} p_j \left[\exp \left(1.2\tau \sum_{i \in \Lambda_{\delta}}
  a_{ij}^2 p_i\right) -1\right] + 1}\\
& \le & 
\prod_{j \in \Lambda_{\delta}^c} p_j\left(1.44 \tau \sum_{i \in \Lambda_{\delta}}
  a_{ij}^2 p_i\right) + 1  \le  
\prod_{j \in \Lambda_{\delta}^c} \exp\left(1.44 \tau
  p_j\sum_{i \in \Lambda_{\delta}} a_{ij}^2 p_i\right) \\
&  = &  
\exp\left(\sum_{j \in \Lambda_{\delta}^c}  1.44 \tau p_j \sum_{i \in \Lambda_{\delta}}
  a_{ij}^2 p_i\right) \le  \exp\left(1.44 \tau  \sum_{i\not=j} a_{ij}^2 p_i p_j\right)
\eens
The lemma thus holds.
\end{proof}

\noindent{\bf Step 5. Putting things together.}

By Jensen's inequality~\eqref{eq::Jensen}, definition of $f(\xi,
\delta,  X_{\Lambda_{\delta}})$ in~\eqref{eq::mgf-subg-cond} and \eqref{eq::mgf-subg},
we have for all $\abs{\lambda} \le\inv{2\sqrt{C_4} \twonorm{A}}$
\bens
\E \exp(\lambda S) 
& \le &
 \E_{\delta}  \E_{\xi}  \E_{X}  \exp(4 \lambda S_{\delta})  \\
& = &
 \E_{\delta}  \E_{\xi}  \E_{X_{\Lambda_{\delta}}} 
\E\left(\exp(4 \lambda S_\delta) |\xi, \delta,
  X_{\Lambda_{\delta}}\right) \\
& = &
 \E_{\delta}  \E_{\xi}  \E_{X_{\Lambda_{\delta}}} 
f(\xi, \delta,  X_{\Lambda_{\delta}})\\
& \le &
 \E_{\delta} \tilde{f}_{\delta}
 \le \exp\left(1.44C_4 \lambda^2 \sum_{i\not=j} a_{ij}^2 p_i p_j\right)
\eens
Thus \eqref{eq::mgfboundS} holds.
\end{proofof2}

%% file: sparseHW.bbl
\begin{thebibliography}{10}

\bibitem{AW15}
R.~Adamczak and P.~Wolff.
\newblock Concentration inequalities for non-lipschitz functions with bounded
  derivatives of higher order.
\newblock {\em Probability Theory and Related Fields}, 162(3):531--586, 2015.

\bibitem{BM12}
F.~Barthe and E.~Milman.
\newblock Transference principles for log-sobolev and spectral-gap with
  applications to conservative spin systems, 2012.

\bibitem{PG99}
V.~H. de~la Pe\~{n}a and E.~Gin\'{e}.
\newblock {\em Decoupling. From dependence to independence. Probability and its
  Applications (New York).}
\newblock Springer-Verlag, New York, 1999.

\bibitem{PM95}
V.~H. de~la Pe\~{n}a and S.~J. Montgomery-Smith.
\newblock Decoupling inequalities for the tail probabilities of multivariate
  u-statistics.
\newblock {\em Ann. Probab.}, 23(2):806–816, 1995.

\bibitem{DKN10}
I.~Diakonikolas, D.~Kane, and J.~Nelson.
\newblock Bounded independence fools degree-2 threshold functions.
\newblock In {\em Proceedings of the 51st Annual IEEE Symposium on Foundations
  of Computer Science (FOCS 2010)}, 2010.

\bibitem{FR13}
A.~Foucart and H.~Rauhut.
\newblock {\em A Mathematical Introduction to Compressive Sensing. Applied and
  Numerical Harmonic Analysis}.
\newblock Birkhauser, 2013.

\bibitem{HW71}
D.~L. Hanson and E.~T. Wright.
\newblock A bound on tail probabilities for quadratic forms in independent
  random variables.
\newblock {\em Ann. Math. Statist.}, 42:1079--1083, 1971.

\bibitem{HJ91}
R.~Horn and C.~Johnson.
\newblock {\em Topics in Matrix Analysis}.
\newblock Cambridge University Press; Reprint edition, 1991.

\bibitem{HKZ12}
D.~Hsu, S.~Kakade, and T.~Zhang.
\newblock Tail inequalities for sums of random matrices that depend on the
  intrinsic dimension.
\newblock {\em Electronic Communications in Probability}, 17:1--13, 2012.

\bibitem{Lat06}
R.~Latala.
\newblock Estimates of moments and tails of gaussian chaoses.
\newblock {\em Ann. Probab}, 34(6):2315–--2331, 2006.

\bibitem{Rud15}
M.~Rudelson.
\newblock On the complexity of the set of unconditional convex bodies.
\newblock {\em Discrete and Computational Geometry}, 2015.
\newblock to appear.

\bibitem{RV13}
M.~Rudelson and R.~Vershynin.
\newblock {H}anson-{W}right inequality and sub-gaussian concentration.
\newblock {\em Electronic Communications in Probability}, 18:1--9, 2013.

\bibitem{Tal95}
M.~Talagrand.
\newblock Sections of smooth convex bodies via majorizing measures.
\newblock {\em Acta. Math.}, 175:273--300, 1995.

\bibitem{VerDecouple}
R.~Vershynin.
\newblock A simple decoupling inequality in probability theory, 2011.
\newblock http://www-personal.umich.edu/~romanv/papers/papers.html.

\bibitem{Vers12}
R.~Vershynin.
\newblock {\em Introduction to the non-asymptotic analysis of random matrices}.
\newblock Cambridge University Press, 2012.
\newblock Chapter 5 of the book Compressed Sensing, Theory and Applications,
  ed. Y. Eldar and G. Kutyniok.

\bibitem{HW73}
E.~T. Wright.
\newblock A bound on tail probabilities for quadratic forms in independent
  random variables whose distributions are not necessarily symmetric.
\newblock {\em Ann. Probab.}, 1:1068--1070, 1973.

\end{thebibliography}
